\pdfoutput=1
\documentclass[paper=a4, english, final ]{scrartcl}

\usepackage{babel}

\usepackage[T1]{fontenc}
\usepackage[utf8]{inputenc}

\usepackage{lmodern}

\usepackage[final, babel ]{microtype}

\usepackage{amsfonts} \usepackage{amssymb}  \usepackage{mathtools} \mathtoolsset{showonlyrefs}
\providecommand\given{} \newcommand\SetSymbol[1][]{\mathrel{}\mathclose{}#1|\allowbreak\mathopen{}\mathrel{}}
\DeclarePairedDelimiterX\Set[1]{\lbrace}{\rbrace}{\renewcommand\given{\SetSymbol[\delimsize]} #1 }

\usepackage{tikz-cd} \tikzset{
	pf/.style={commutative diagrams/.cd, every arrow, every label},
	surj/.style=commutative diagrams/two heads,
	inj/.style=commutative diagrams/hook,
	gl/.style=commutative diagrams/equal,
	mat/.style={matrix of math nodes, commutative diagrams/.cd, every cell},
	dr/.style={matrix of math nodes, commutative diagrams/.cd, every cell, column sep=small},
	seq/.style={matrix of math nodes, commutative diagrams/.cd, every cell, column sep=small}
	}
\newenvironment{diag}{\begin{equation}\begin{tikzpicture}[commutative diagrams/.cd, every diagram, baseline=(current bounding box.center)]}{\end{tikzpicture}\end{equation}\ignorespacesafterend}

\usepackage{amsthm} \theoremstyle{plain}

\newtheorem{prop}[equation]{Proposition}

\newtheorem*{thm*}{Theorem}
\newtheorem*{prop*}{Proposition}

\newtheorem*{principle*}{Principle}

\theoremstyle{definition}
\newtheorem{cor}[equation]{Corollary}
\newtheorem{lemma}[equation]{Lemma}
\newtheorem{defn}[equation]{Definition}

\newtheorem*{cor*}{Corollary}
\newtheorem*{lemma*}{Lemma}
\newtheorem*{defn*}{Definition}

\theoremstyle{remark}
\newtheorem{rem}[equation]{Remark}
\newtheorem{ex}[equation]{Example}

\newtheorem*{rem*}{Remark}
\newtheorem*{ex*}{Example}

\usepackage{thmtools, thm-restate}

\usepackage{csquotes}
\usepackage[style=authoryear, date=year,
	isbn=false,
	]{biblatex}

\usepackage{chngcntr}

\author{Enno Keßler \and Yuri I. Manin \and Yingying Wu}
\title{Moduli spaces of SUSY curves\\ and their operads}
\dedication{To Shiing-Shen Chern, the great teacher of mathematics}
\addtokomafont{dedication}{\normalsize}
\date{}

\usepackage{faktor}

\usepackage{enumitem}
\setlist[enumerate,1]{label = (\roman*)}

\usepackage[final, pdfusetitle ]{hyperref}
\hypersetup{colorlinks=false}

\newcommand{\cat}[1]{\mathsf{#1}}
\DeclareMathOperator{\id}{id}
\DeclareMathOperator{\Ber}{Ber}
\DeclareMathOperator{\Aut}{Aut}
\DeclareMathOperator{\Iso}{Iso}

\newcommand{\ic}{\mathrm{i}}

\newcommand{\Z}{\mathbb{Z}}
\newcommand{\C}{\mathbb{C}}

\newcommand{\cD}{\mathcal{D}}
\newcommand{\cO}{\mathcal{O}}
\newcommand{\cT}{\mathcal{T}}

\counterwithin{equation}{subsection}
\counterwithin{figure}{section}

\begin{document}
\maketitle
\begin{abstract}
	This article generalizes the operad of moduli spaces of curves to SUSY curves.
	SUSY curves are algebraic curves with additional supersymmetric or supergeometric structure.
	Here, we focus on the description of the relevant category of graphs and its combinatorics as well as the construction of dual graphs of SUSY curves and the SUSY operad taking values in a category of moduli spaces of SUSY curves with Neveu--Schwarz and Ramond punctures.
\end{abstract}

\section{Introduction and summary}\label{Sec:IntroductionAndSummary}
In the earlier works (cf.~\cite{LM-EMO} and references therein), it was shown that operadic constructions developed in~\cite{G-OMSG0RS} can be extended to include additional data on algebraic curves.
In~\cite{BM-GOIC} a general formalism of labeled graphs and generalized operads was developed.
Here, we combine both techniques to construct an operad that encodes geometric properties of moduli spaces of SUSY curves with Neveu--Schwarz and Ramond punctures.

SUSY curves, also called super Riemann surfaces or super curves, are supergeometric generalizations of algebraic curves with spin structure that have been introduced in the context of superstring theory, see~\cite{F-NSTTDCFT}.
Their most distinctive feature for the context of this work is that families of SUSY curves may develop two types of nodes as was argued in~\cite{D-LaM}.
The two types of nodes are now called Neveu--Schwarz and Ramond nodes and one also considers markings of those two types: Neveu--Schwarz punctures and Ramond punctures.
Compact moduli spaces \(\overline{\mathcal{M}}_{g,I_{NS},I_R}\) of stable SUSY curves of genus \(g\) with  Neveu--Schwarz labeled by the finite set \(I_{NS}\) and Ramond punctures labeled by the finite set \(I_R\) have been constructed in~\cite{FKP-MSSCCLB} as smooth Deligne--Mumford superstacks.

To achieve a combinatorial description of stable SUSY curves and the boundary strata of the moduli stack~\(\overline{\mathcal{M}}_{g,I_{NS},I_R}\), we introduce in Definition~\ref{defn:SUSYGraph} below the category of stable SUSY graphs \(\cat{SGr^{st}}\).
Stable SUSY graphs are stable modular graphs with an additional coloring of the edges and tails to encode if nodes or markings are of Neveu--Schwarz or Ramond type.
In Definition~\ref{defn:DualGraph}, we show how a stable SUSY curve has a stable SUSY graph as dual graph.

SUSY operads in the framework of~\cite{BM-GOIC} are functors \(\cat{SGr^{st}}\to \cat{C}\) that send graftings to isomorphisms.
SUSY operads generalize modular operads that have been introduced in~\cite{GK-MO}.
In particular, the modular operad \(\cat{o}\colon \cat{MGr^{st}}\to \cat{M}\) sending modular graphs to products of moduli stacks of algebraic curves and contractions to gluings is known to encode important combinatorial data for moduli stacks of curves.
The main result of the paper is the following supergeometric generalization:

\begin{restatable*}{thm}{Operad}\label{thm:SUSYOperad}
	There is a category \(\cat{SM}\) such that the map
	\begin{equation}
		\tau \mapsto \prod_{v\in V_\tau} \overline{\mathcal{M}}_{g(v), F_{\tau, NS}(v), F_{\tau, R}(v)}
	\end{equation}
	can be extended to an operad \(\cat{O}\colon \cat{SGr^{st}}\to \cat{SM}\) which sends graftings to the identity, isomorphisms of corollas to isomorphisms of the moduli superstacks, and contractions of edges of Neveu--Schwarz resp.\ Ramond type to gluings of punctures of Neveu--Schwarz resp.\ Ramond type.

	Furthermore, there is a functor \(\cat{P}\colon \cat{SM}\to \cat{M}\) such that \(\cat{o}\circ \cat{F} = \cat{P} \circ \cat{O}\) for the forgetful functor \(\cat{F}\colon \cat{SGr^{st}}\to \cat{MGr^{st}}\).
\end{restatable*}

Using this Theorem, one obtains a descriptions of the moduli space of SUSY curves with prescribed dual graph in terms of gluing data and its dimension.
This is a generalization of a description of \(\overline{\mathcal{M}}_{0,I_{NS},\emptyset}\) via trees and as a stratified superorbifold in~\cite{KSY-SQCI}.

The structures, upon which we focus in this study, are partly motivated by the study of \enquote{phylogenetic trees} in~\cite{WY-CTPSMSC}.
Phylogenetic trees are used to model evolution and are mathematically described as trees where edges have a length.
Y. Wu and S.-T. Yau describe combinatorially various patterns of degeneration of generically smooth projective curves of genus zero with marked points at the boundaries of moduli spaces of such curves using phylogenetic trees where the length of edges may decrease to zero.
The results in Section~\ref{Sec:StableSUSYTrees} indicate how moduli spaces of phylogenetic trees could be generalized to moduli spaces of \enquote{phylogenetic SUSY trees}.

The paper is organized as follows:
We start the Section~\ref{Sec:SUSYGraphs} with a systematic description of the formalism of graphs as a language for describing the combinatorics of degeneration of various types of algebraic/analytic curves and SUSY curves.
The Section~\ref{Sec:StableSUSYCurvesWithPuncturesAndTheirDualGraph} recalls the notions of stable SUSY curve with Neveu--Schwarz and Ramond punctures from~\cite{FKP-MSSCCLB} and describes how a stable SUSY graph is obtained from stable SUSY curves.
Finally, Section~\ref{Sec:SupermodularOperads} contains the main new results of this paper, the construction of the operad of moduli superstacks of SUSY curves, its symmetry properties and relationship to the operad of moduli stacks of curves.

\subsection*{Acknowledgments}
We thank Katherine Maxwell, Vincentas Mulevičius, Lukas Müller and Luuk Stehouwer as well as the members of the Seminar for Mathematics of Quantum Field Theory at Max-Planck-Institut für Mathematik for useful comments and discussions.
Enno and Yuri are deeply grateful to Marianne and Xenia for all their love.

\section{SUSY Graphs}\label{Sec:SUSYGraphs}

\subsection{Graphs, their morphisms and categories}
We accept the general framework of~\cite{BM-GOIC}, stressing the difference between the formal definitions of \emph{graphs} on the one side, and their \emph{geometric realizations} on the other side.

A \emph{graph} $\tau$ is a family of structures (just sets or structured sets and maps between them in the simplest cases) $(F_{\tau}, V_{\tau}, \partial_{\tau} , j_{\tau})$.
Elements of $F_{\tau}$, resp. $V_{\tau}$, are called \emph{flags}, resp. \emph{vertices} of $\tau$.
The map $\partial_{\tau}\colon F_{\tau} \to V_{\tau}$ associates to each flag a vertex, called its \emph{boundary}.
For a given vertex \(v\in V_\tau\) we denote by \(F_\tau(v) = {\left(\partial_\tau\right)}^{-1}(v)\subset F_\tau\) the set of flags adjacent to \(v\).

The map $j_{\tau}\colon F_{\tau} \to F_{\tau}$ must satisfy the condition $j_{\tau}^2 = \id$, identical map of $F_{\tau}$ to itself.
Flags fixed by \(j_\tau\) are called tails; we denote the set of tails of \(\tau\) by \(T_\tau\).
Two-element orbits of \(j_\tau\) are called edges of \(\tau\); we denote the set of edges of \(\tau\) by \(E_\tau\).

For a discussion of their geometric meaning, and of many marginal (\enquote{degenerate}) cases, see~\cite{BM-GOIC}.

A \emph{morphism of graphs} $h\colon \tau \to \sigma $ is a triple of maps
\begin{equation}
	\left(h^F\colon F_{\sigma} \to F_{\tau},
\quad  h_V\colon V_{\tau} \to V_{\sigma},
\quad  j_h\colon F_{\tau} \setminus h^F (F_{\sigma}) \to F_{\tau} \setminus h^F (F_{\sigma}) \right),
\end{equation}
where $h^F$ is an injective contravariant map, $h_V$ is a surjective covariant map, and $j_h$ is an involution, satisfying a list of additional restrictions:
cf.~\cite[Definition~1.2.1]{BM-GOIC}.

Thus defined, graphs form a category $\cat{Gr}$, with symmetric monoidal structure corresponding to the disjoint union \(\coprod\) of geometric realizations of graphs.

\begin{rem}
	The definition of \(\cat{Gr}\) in this article is particularly suited for the study of dual graphs of curves and differs from other texts, for example,~\cite{BH-MSNPC}.
	In particular, graphs in \(\cat{Gr}\) can have tails, multiple but undirected edges, loops and disconnected components.
	Graph morphisms allow \emph{graftings}, that is connecting two tails to form an edge, \emph{contractions} of edges and \emph{virtual contractions}, that is contraction of a pair of tails, as well as \emph{mergers} of vertices preserving flags.
	But on the other hand, morphisms cannot break edges into tails and inclusions of subgraphs are not necessarily graph morphisms because the map \(h_V\) between vertices is required to be surjective.
\end{rem}

\begin{rem}
	It is pretty clear, that in the definitions above one may replace (structured) sets by objects of a category, maps of sets by morphisms between objects of this category etc.
	Of course, one then should take care about various compatibility restrictions, lifted to the level of a (small) category.
\end{rem}

\subsection{Labeled graphs}
In what follows we will need graphs with different labelings.
Abstractly, a category of labeled graphs \(\cat{\Gamma}\) comes with a functor \(\cat{\psi\colon \Gamma\to Gr}\) satisfying several properties, see~\cite[1.3 Definition]{BM-GOIC}.
As mentioned there, a labeled graph \(\sigma\in\cat{\Gamma}\) can be imagined as the underlying graph \(\psi(\sigma)\) together with some additional data on vertices, flags or edges.
Some examples:

\begin{ex}[modular graphs, see Example 1.3.2c) in~\cite{BM-GOIC}]\label{ex:ModularGraphs}
	A \emph{modular graph} is a graph \(\tau=(F_\tau, V_\tau, \partial_\tau, j_\tau)\) together with a \emph{genus labeling}, that is, is a map \(g\colon V_\tau\to \Z_{\geq 0}\).
	A morphism of modular graphs is a morphism of graphs such that the genus of a vertex in the image is given by the sum of the genera of the vertices in the preimage plus the number of contracted loops at that vertex.
	Any two preimages of a vertex must be connected by a path of contracted edges; in particular mergers are not allowed.

	The genus of a connected modular graph \(\tau\) is given by
	\begin{equation}
		g_\tau = \sum_{v\in V_\tau} (g(v) - 1) + \#E_\tau + 1,
	\end{equation}
	where \(\#E_\tau\) is the cardinality of the set of edges of \(\tau\).
	The genus of disconnected modular graphs is the sum of the genera of its connected components.

	A \emph{tree} is a genus labeled graph of genus zero.
	We say that the modular graph \(\tau\) is \emph{stable} if for every vertex \(v\in V_\tau\) it holds \(2g(v)-2 + \#F_\tau(v)>0\).
\end{ex}

\begin{ex}[colored graphs, see Example 1.3.2.d) in~\cite{BM-GOIC}]\label{ex:ColoredGraphs}
	A \emph{colored graph} is a graph where edges and tails are assigned a color.
	More precisely, an \(I\)-coloring on \(\tau=(F_\tau, V_\tau, \partial_\tau, j_\tau)\) is a map \(c\colon F_\tau\to I\) to some fixed-finite set \(I\) whose elements are called colors such that \(c\circ j_\tau = c\).
	Morphisms \(h\colon \tau \to \sigma\) of \(I\)-colored graphs preserve the color of edges, that is \(c_\tau\circ h^F = c_\sigma\) and \(c_\tau\circ j_h = c_\tau|_{F_\tau\setminus h^F(F_\sigma)}\).
\end{ex}

\begin{ex}[graphs with labeled tails]\label{ex:ILabeledGraph}
	Let \(I\) be a finite set.
	An \emph{\(I\)-labeling of the tails} on a graph \(\tau\) is given by a bijective map \(l_\tau\colon I\to T_\tau\).
	A morphism \(h\colon \tau\to \sigma\) between \(I\)-labeled graphs is bijective on tails and preserves the labeling, that is, \(l_\tau\circ h^F|_{T_\sigma} = l_\sigma\).

	Any two \(I\)-labelings \(l_\tau\) and \(l'_\tau\) on \(\tau\) differ by an element of the group \(\Aut I\) of bijections \(I\to I\).
	Any bijection \(I\to J\) between finite sets yields a functor from \(I\)-labeled to \(J\)-labeled graphs.

	It is often convenient to label by the finite set \(I=\Set{1,\dotsc, k}\) for some positive integer \(k\), in which case one speaks about a \(k\)-labeling.
	The tails of a \(k\)-labeled graph are numbered and two \(k\)-labeled graphs differ by an element of the symmetric group in \(k\)-elements.
	Every finite set \(I\) of cardinality \(\#I=k\) is non-canonically isomorphic to \(\Set{1,\dotsc, k}\) but in this article we need to keep track of isomorphisms of finite sets explicitly.
	See also Remark~\ref{rmk:kLabeling} below.
\end{ex}

\begin{defn}
	Let \(\cat{MGr}\) be the category of modular graphs and morphisms compatible with the genus labeling in the sense of Example~\ref{ex:ModularGraphs}.
	We also consider the full subcategories of stable graphs~\(\cat{MGr^{st}}\), modular trees~\(\cat{MTr}\) and stable modular trees~\(\cat{MTr^{st}}\).

	Furthermore, we denote by \(\cat{MGr}^{\cat{st}}_{g,I}\) the category of stable modular graphs of genus \(g\) with \(I\)-labeled tails.
	Morphisms in \(\cat{MGr}^{\cat{st}}_{g,I}\) respect the genus labeling and the labeling of tails in the sense of Examples~\ref{ex:ModularGraphs} and~\ref{ex:ILabeledGraph}.
\end{defn}
The categories \(\cat{MGr}\), \(\cat{MGr^{st}}\), \(\cat{MTr}\) and \(\cat{MTr^{st}}\) are symmetric monoidal categories with respect to disjoint union.

\subsection{Graphs relevant to the encoding of SUSY curves}\label{SSec:GraphsRelevantToTheEncodingOfSUSYCurves}
\begin{defn}\label{defn:SUSYGraph}
	A \emph{SUSY graph} is a graph $\tau=(F_{\tau},  V_{\tau}, \partial_{\tau}, j_{\tau})$, endowed with
	\begin{enumerate}
		\item\label{item:defnSUSYGraph:GenusLabeling}
			a genus labeling $g_{\tau}\colon V_{\tau} \to \Z_{\ge 0}$ and
		\item\label{item:defnSUSYGraph:Coloring}
			a coloring \(c_\tau\colon F_\tau\to \Set{NS, R}\) such that for every vertex \(v\in V_\tau\) the number of adjacent flags with color R must be even.
	\end{enumerate}
	A morphism of SUSY graphs is a morphism between graphs that preserves the genus labeling as well as the coloring of flags, compare Examples~\ref{ex:ModularGraphs},~\ref{ex:ColoredGraphs}.
	We denote the category of SUSY graphs by \(\cat{SGr}\).

	We denote the full subcategory of \(\cat{SGr}\) where objects are stable SUSY graphs by \(\cat{SGr^{st}}\).
	The full subcategory of \(\cat{SGr}\) where objects are trees is denoted by \(\cat{STr}\) and the category of stable SUSY trees by \(\cat{STr^{st}}\).
\end{defn}

The coloring \(c_\tau\) induces partitions of the set of all flags \(F_\tau = F_{\tau, NS} \cup F_{\tau, R}\) of a SUSY graph, the set of flags adjacent to a vertex \(F_{\tau}(v) = F_{\tau, NS}(v) \cup F_{\tau, R}(v)\), tails \(T_\tau = T_{\tau, NS}\cup T_{\tau, R}\) and edges \(E_\tau = E_{\tau, NS}\cup E_{\tau, R}\) into those with color \(NS\) and those with color \(R\).
We denote elements of \(F_{\tau, NS}\) Neveu--Schwarz flags and elements of \(F_{\tau, R}\) Ramond flags which we sometimes abbreviate as NS-flags and R-flags respectively.
Similarly for edges and tails.

Notice that the number of Ramond tails of a SUSY graph is always even.
The number of Neveu--Schwarz tails (resp.\ Ramond tails) of the image of a morphisms of SUSY graphs may be lower than the number of Neveu--Schwarz tails (resp.\ Ramond tails) of the domain.
Morphisms between SUSY graphs can graft tails of the same color and virtually contract pairs of tails of the same color.

There is a forgetful functor \(\cat{F}\colon \cat{SGr}\to \cat{MGr}\) that restricts to forgetful functors \(\cat{SGr^{st}}\to\cat{MGr^{st}}\), \(\cat{STr}\to \cat{MTr}\) and \(\cat{STr^{st}}\to\cat{MTr^{st}}\).
There is also a functor \(\cat{I}\colon \cat{MGr}\to \cat{SGr}\) sending a modular graphs to the same graph where every flag is of color~\(NS\).
The functor \(\cat{I}\) embeds \(\cat{MGr}\) as a full subcategory in \(\cat{SGr}\) and satisfies \(\cat{F}\circ \cat{I}=\cat{id_{MGr}}\).
The restrictions of \(\cat{I}\) to \(\cat{MGr^{st}}\to \cat{SGr^{st}}\), \(\cat{MTr}\to \cat{STr}\) and \(\cat{MTr^{st}}\to \cat{STr^{st}}\) are also full subcategories.

The categories \(\cat{SGr}\), \(\cat{SGr^{st}}\), \(\cat{STr}\) and \(\cat{STr^{st}}\) are symmetric monoidal categories with respect to the disjoint union \(\coprod\) of SUSY graphs.
	Both the inclusion functors \(\cat{I}\) and the forgetful functors \(\cat{F}\) are symmetric monoidal functors.

\begin{defn}
	Choose a pair of finite sets \((I_{NS}, I_R)\), such that the cardinality of \(I_R\) is even.
	An \emph{$(I_{NS},I_R)$-labeled SUSY} graph is a SUSY graph \((\tau, g_\tau, c_\tau)\) together with two separate labelings of NS-tails and  R-tails, that is, bijections
	\begin{align}
		l_{\tau , NS}\colon I_{NS} &\to T_{\tau , NS}, &
		l_{\tau , R}\colon I_R &\to T_{\tau , R}.
	\end{align}
	A morphism of \((I_{NS}, I_R)\)-labeled SUSY graphs is a morphism \(h\colon \tau\to \sigma\) of SUSY graphs that is bijective on Neveu--Schwarz tails and Ramond tails and preserve the labelings, that is \(l_{\tau, NS}\circ h^F|_{T_{\sigma, NS}} = l_{\sigma, NS}\) and \(l_{\tau, R}\circ h^F|_{T_{\sigma, R}} = l_{\sigma, R}\).

	We denote the category of \((I_{NS}, I_R)\)-labeled stable SUSY graphs of genus \(g\) by \(\cat{SGr}^{\cat{st}}_{g, I_{NS}, I_R}\).
\end{defn}

Any graph \(\tau\) is \((T_{\tau, NS}, T_{\tau, R})\)-labeled.
The forgetful functor \(\cat{F}_{g, I_{NS}, I_R}\colon \cat{SGr}^{\cat{st}}_{g,I_{NS}, I_R}\to \cat{MGr}^{\cat{st}}_{g, I_{NS}\cup I_R}\) and inclusion functor \(\cat{I}_{g, I}\colon \cat{MGr}^{\cat{st}}_{g,I}\to \cat{SGr}^{\cat{st}}_{g, I, \emptyset}\) are constructed in the obvious way.
Here and in the following, we assume that \(I_{NS}\cup I_R\) is a disjoint union, compare also~\cite[Section~1.2.4]{BM-GOIC}.

Any two \((I_{NS}, I_R)\)-labelings \((l_{\tau, NS}, l_{\tau, R})\) and \((l'_{\tau, NS}, l'_{\tau, R})\) on a SUSY graph \(\tau\) differ by a pair of automorphisms of finite sets \(s_{NS}\colon I_{NS}\to I_{NS}\) and \(s_R\colon I_R\to I_R\) such that
\begin{align}
	l'_{\tau, NS} &= l_{\tau, NS} \circ s_{NS}, &
	l'_{\tau, R} &= l_{\tau, R} \circ s_{R}.
\end{align}
The group \(\Aut(I_{NS}, I_R)\) of automorphisms of the pair \((I_{NS}, I_R)\) of finite sets is also the automorphism group of the unlabeled SUSY corolla with Neveu--Schwarz tails \(I_{NS}\) and Ramond tails \(I_R\).
Remember that a corolla is a graph with a single vertex.

More generally, any pair of bijections \((s_{NS}\colon I_{NS}\to J_{NS}, s_R\colon I_R\to J_R)\) yields a functor \(\cat{a}_s\colon \cat{SGr}^{\cat{st}}_{g, I_{NS}, I_R} \to \cat{SGr}^{\cat{st}}_{g, J_{NS}, J_R}\) which is an equivalence of categories.

\subsection{Stable SUSY trees}\label{Sec:StableSUSYTrees}
Recall that stable SUSY trees are SUSY graphs of genus zero such that every vertex has at least three adjacent flags.
The forgetful functor \(\cat{F}\colon \cat{STr^{st}}\to\cat{MTr^{st}}\) sends a stable SUSY tree to a stable modular tree.
Conversely, we will now show that a modular tree can be given the structure of a SUSY graph by coloring the tails.
The color of the edges can be inferred using the parity condition at every vertex.

\begin{lemma}\label{lemma:SUSYLiftOfModularGraphs}
	Let \(\sigma\in \cat{MGr}^{\cat{st}}_{0,I}\) be stable modular tree with \(I\)-labeled tails and \(I=I_{NS}\cup I_R\) a partition such that the cardinality of \(I_R\) is even.
	There exists a unique stable \((I_{NS}, I_R)\)-labeled SUSY tree \(\tau\in\cat{SGr}^{\cat{st}}_{0,I_{NS}, I_R}\) such that \(\cat{F}_{0,I_{NS}, I_R}(\tau)=\sigma\).
\end{lemma}
\begin{proof}
	We need to construct the partition \(c_\tau\) of flags into Neveu--Schwarz and Ramond flags.
	The partition of the tails is induced by the \(I\)-labeling and the partition \(I=I_{NS}\cup I_R\).

	The partition of the edges is uniquely obtained from the partition of the tails together with the parity condition by induction over the number of edges:
	If the tree~\(\tau\) has no edges there is nothing to show.
	Every tree~\(\tau\) with at least one edge has a vertex \(v\) which is the boundary of precisely one edge \(e\).
	Stability implies that \(v\) has at least two tails.
	If the number of Ramond tails bounding to \(v\) is even the edge \(e\) needs to be be a Neveu--Schwarz edge.
	If the number of Ramond tails bounding \(v\) is odd the edge \(e\) needs to be an Ramond edge.
	We can now proceed by inductively considering the tree \(\tau'\) obtained from \(\tau\) by cutting the edge \(e\) and deleting the vertex \(v\) as well as its adjacent flags.
\end{proof}

There are \(\binom{k}{k_R}\) choices of partitions \(I=I_{NS}\cup I_R\) such that \(k=\#I\) as well as \(k_R=\#I_R\) are fixed and \(2^{k-1}=\sum_i\binom{k}{2i}\) choices of partitions \(I=I_{NS}\cup I_R\) such that the cardinality of \(I_R\) is even.
That is, \(\sigma\in\cat{MTr^{st}}\) has \(2^{\#T_\sigma-1}\) preimages under \(\cat{F}\colon \cat{STr^{st}}\to \cat{MGr^{st}}\).

\section{Stable SUSY curves with punctures and their dual graph}\label{Sec:StableSUSYCurvesWithPuncturesAndTheirDualGraph}
In this section we recall the notions of stable SUSY curves with punctures and their respective moduli spaces.
Furthermore we give the construction of the dual graph of a SUSY curve with punctures as a SUSY graph.

The definition of SUSY curves or super Riemann surfaces appeared first in the context of super string theory, see, for example,~\cites{F-NSTTDCFT}.
SUSY curves are superschemes of dimension \(1|1\) with an additional structure and generalize, in many aspects, algebraic curves to supergeometry.
Particular examples of SUSY curves can be constructed from algebraic curves together with a spinor bundle.

Early studies of the moduli spaces of SUSY curves are~\cites{LBR-MSRS}{CR-SRSUTT}.
It was argued in~\cite{D-LaM} that families of SUSY curves may degenerate in SUSY curves with two distinct types of nodes, called Neveu--Schwarz (\(NS\)) nodes and Ramond (\(R\)) nodes.
While Neveu--Schwarz nodes are transversal intersections of SUSY curves, Ramond nodes have an additional degeneration of the spinor bundle.

The role of marked points on purely even algebraic curves is played by \enquote{punctures} on a SUSY curve.
Punctures are likewise divided into two classes, called Neveu--Schwarz and Ramond punctures.
For a family of stable SUSY curves $M \to B$, Neveu--Schwarz punctures are given by sections $B\to M$, whereas Ramond punctures are relative Cartier divisors $\mathcal{R}\subset M$ such that the projection $\mathcal{R}\to B$ is smooth of dimension $0|1$ together with additional structural conditions.

For modern accounts on stable SUSY curves with punctures and their moduli spaces we refer to~\cites{FKP-MSSCCLB}{OV-SMSGZSUSYCRP}{BR-SMSCRP}.

\subsection{SUSY curves with punctures}\label{SSec:SUSYCurvesWithPunctures}
We assume familiarity with algebraic supergeometry as, for example, in~\cite{M-GFTCG}.
The following definition of SUSY curve with punctures is taken from~\cite[Definition~2.3 in Section~2.2]{FKP-MSSCCLB}.
\begin{defn}\label{defn:SUSYCurveWithPunctures}
	A \emph{SUSY curve} over the base superscheme \(B\) with Neveu--Schwarz punctures labeled by the finite set \(I_{NS}\) and Ramond punctures labeled by the finite set~\(I_R\) is a tuple \((M, \{s_i\}, \mathcal{R}, \mathcal{D})\) where
	\begin{itemize}
		\item
			\(\pi\colon M\to B\) is a smooth morphism of superschemes of relative dimension~\(1|1\) generic fibers,
		\item
			\(s_i\colon B\to M\) for \(i\in I_{NS}\) are sections of \(\pi\) such that their reductions are distinct, called Neveu--Schwarz punctures,
		\item
			\(\mathcal{R}\) is an unramified relative effective Cartier divisor of codimension \(0|1\) in \(M\) of degree \(\#I_R\) whose labeled components~\(r_j\), \(j\in I_R\) are called the Ramond punctures
		\item
			the line bundle \(\mathcal{D}\) is a subbundle \(\mathcal{D}\subset{\cT_M}\) of rank \(0|1\) such that the commutator of vector fields induces an isomorphism
			\begin{equation}
				\mathcal{D}\otimes \mathcal{D}\to \left(\faktor{\cT_M}{\mathcal{D}}\right)(-\mathcal{R}).
			\end{equation}
	\end{itemize}
\end{defn}

Considering the quotient of the structure sheaf of \(M\) by the ideal of nilpotent elements one obtains the reduction \(i_{red}\colon M_{red}\to M\).
The family \(M_{red}\to B_{red}\) is a smooth family of curves over \(B_{red}\).
The reduction of the Neveu--Schwarz punctures \(s_i\) yields marked points of \(M_{red}\).
The reduction of the Ramond punctures yields a divisor \(\mathcal{R}_{red}\) which are equivalent to further marked points of \(M_{red}\).
That is, we can see \(M_{red}\to B_{red}\) as a smooth family of curves with marked points labeled by \(I_{NS}\cup I_R\).

But it is also interesting to make a distinction between the Neveu--Schwarz punctures and the Ramond punctures in the reduced case:
The pullback \(S=i^*\mathcal{D}\) is a spinor bundle on \(M_{red}\) with a degeneration over \(\mathcal{R}_{red}\), or more precisely
\begin{equation}
	S\otimes S = \cT_{M_{red}} (-\mathcal{R}_{red})
\label{eq:DegeneratedSpinorBundle}\end{equation}
The family \(M_{red}\to B_{red}\) together with the spinor bundle \(S\) as in Equation~\eqref{eq:DegeneratedSpinorBundle} is called a spin curve with punctures of type \(0\) labeled by \(I_{NS}\) and punctures of type \(1\) labeled by \(I_R\) in~\cite{AJ-MTSC}.

A SUSY curve over \(B= pt\) is equivalent to the data of \((M_{red}, S, \Set{s_{i,red}}, \mathcal{R}_{red})\) satisfying Equation~\eqref{eq:DegeneratedSpinorBundle}.

If \(M\) is irreducible of genus \(g\) the Theorem of Riemann Roch implies
\begin{equation}
	2\deg S
	= \deg \cT_{M_{red}} - \deg \mathcal{R}_{red}
	= 2 - 2g - \deg \mathcal{R}_{red}.
\end{equation}
that is, the number of Ramond punctures needs to be even.

\begin{rem}
	SUSY curves are sometimes also called super Riemann surfaces or super curves, especially if the number of Ramond punctures is zero.
	We will use in this work exclusively the name SUSY curve.
\end{rem}

\subsection{Stable SUSY curves with punctures}
Intuitively, a stable SUSY curve is a generalization of SUSY curves with punctures to include nodes such that its reduced space is a stable curve.
To make this precise, it is necessary to reformulate the non-integrability condition of \(\cD\) as was argued in~\cite{D-LaM}.

Dualizing the exact sequence $0\to \cD \to \cT_{M/B} \to \faktor{\cT_{M/B}}{\cD}=\cD^{\otimes 2}(R) \to 0$, we get $0\to \cD^{-2} (-R) \to \Omega_{M/B} \to \cD^{-1} \to 0$.
This shows that $\cD^{-1}(-R)$ is isomorphic to the Berezinian of $\Omega_{M/B}$, denoted $\omega_{M/B}=\Ber\Omega_{M/B}$, and in turn produces the derivation $\delta\colon \cO_M \to \omega_{M/B}(R)$, trivial on lifts of $\cO_B$.
The data of \(\delta\) is equivalent to the data of the line bundle \(\cD\).

\begin{defn}
	Consider a superscheme $B$.
	A family of \emph{stable SUSY curves with punctures over $B$} is a tuple \((M, \{s_i\}, \mathcal{R}, \delta)\) consisting of
	\begin{itemize}
		\item
			a proper, flat and relatively Cohen--Macaulay superscheme $\pi\colon M\to B$
		\item
			\(s_i\colon B\to M\) for \(i\in I_{NS}\) are sections of \(\pi\) such that their reductions are different, called Neveu--Schwarz punctures,
		\item
			\(\mathcal{R}\) is an unramified relative effective Cartier divisor of codimension \(0|1\) in \(M\) of degree \(\#I_R\) whose labeled components~\(r_j\), \(j\in I_R\) are called the Ramond punctures
		\item
			A derivation $\delta\colon \cO_M \to \omega_{M/B} (\mathcal{R})$, trivial on lifts of $\cO_B$ to $M$.
	\end{itemize}
	These data must satisfy the following conditions:
	\begin{enumerate}
		\item
			$M$ contains an open fiberwise dense subset $U$, such that $U/B$ is smooth of relative dimension $1|1$, and all $s_i$ and $r_i$ are contained in $U$.
		\item
			The tuple \((U, \{s_i\}, \mathcal{R}, \delta)\) is equivalent to a SUSY curve in the sense of Definition~\ref{defn:SUSYCurveWithPunctures}.
		\item
			The reduction \(M_{red}\to B_{red}\) is a stable family of marked curves.
	\end{enumerate}
\end{defn}

By the local theory of the nodes of stable SUSY curves developed in~\cite{D-LaM}, we know that there is a proper SUSY curve \(\tilde{M}\) with Neveu--Schwarz and Ramond punctures together with a map \(\pi\colon \tilde{M}\to M\) of stable SUSY curves such that
\begin{itemize}
	\item
		\(\pi|_{\pi^{-1}(U)}\) is an is an isomorphism of SUSY curves,
	\item
		\(\pi_{red}\colon \tilde{M}_{red}\to M_{red}\) is the normalization,
	\item
		and the preimage of a node consists of two marked points of the same type.
\end{itemize}

Notice that, in general, \(M\) and \(\tilde{M}\) can be reducible and consist of several components~\(M_i\) and \(\tilde{M}_i\) respectively.
We call the punctures of \(\tilde{M}_i\) special points of \(M_i\).
A special point of \(M_i\) is either a puncture of \(M\) or represents a node of \(M\).
Every special point is either a Neveu--Schwarz special point or a Ramond special point.
The number of Ramond special points on an irreducible component is even.

The reduction \({M}_{red}\) of a stable SUSY curve \(M\) is a stable algebraic curve of genus~\(g\) with marked points labeled by \(I_{NS}\cup I_R\).
The pullback \(S=i_{red}^*\cD\) of \(\cD\) along the reduction map \(i_{red}\colon M_{red}\to M\) is a locally free sheaf of rank one with a morphism
\begin{equation}\label{eq:TwistedSpinorBundle}
	S\otimes S\to \cT_{M_{red}}\left(-\mathcal{R}_{red}\right)
\end{equation}
which is an isomorphism away from the nodes.

\subsection{The dual graph of a stable SUSY curve with punctures}
To a given algebraic curve with marked points one associates a dual graph, see~\cites{KM-GWCQCEG}[Chapter~X, §2]{ACG-GACII}{CM-SGZMO}.
We are going to generalize the concept of dual graph to SUSY curves by adding additional labels to edges and tails as follows:

\begin{defn}\label{defn:DualGraph}
	Let \((M, \{s_i\}, \mathcal{R}, \mathcal{D})\) be a stable SUSY curve of genus \(g\) with Neveu--Schwarz punctures labeled by \(I_{NS}\) and Ramond punctures labeled by \(I_R\) over \(B\) such that \(B_{red}=pt\).
	The \emph{dual graph} of \((M, \{s_i\}, \mathcal{R}, \mathcal{D})\) consists of a \((I_{NS}, I_R)\)-labeled SUSY graph \(\tau\in \cat{SGr}^{\cat{st}}_{g,I_{NS},I_R}\) where
	\begin{enumerate}
		\item\label{item:defnDualGraph:Vertices}
			The set of vertices \(V_\tau\) is the set of irreducible components of \(M\).
		\item\label{item:defnDualGraph:Flags}
			The vertex \(v\in V_\tau\) representing the irreducible component \(M_i\) has an adjacent flag for each special point of \(M_i\).
			This determines \(F_\tau\) and \(\partial_\tau\).
		\item\label{item:defnDualGraph:Edges}
			The map \(j_\tau\) maps flags associated to punctures to itself and flags associated to nodes to the flag that represents the same node on the intersecting irreducible component.
		\item\label{item:defnDualGraph:Genus}
			The genus labeling sends the vertex \(v\) representing the irreducible component \(M_i\) to the genus of \(M_i\).
			The stability of the SUSY curve implies that \(\tau\) is a stable graph.
		\item\label{item:defnDualGraph:Coloring}
			The flags~\(F_\tau\) are partitioned into Neveu--Schwarz flags~\(F_{\tau,NS}\) and Ramond flags~\(F_{\tau, R}\) according to the special point they represent.
			As the number of Ramond special points on each irreducible component is even, we have that for every vertex \(v\) the number \(\#F_{\tau, R}(v)\) of Ramond flags at \(v\) is even.
		\item\label{item:defnDualGraph:Punctures}
			The labeling \(l_{\tau, NS}\) (resp. \(l_{\tau, R}\)) sends \(i\in I_{NS}\) (resp. \(j\in I_R)\) to the tail corresponding to the Neveu--Schwarz puncture with label \(i\) (resp. Ramond puncture with label \(j\)).
	\end{enumerate}
\end{defn}

Let \(b\colon B'\to B\) be a map between superschemes such that \(B_{red}=B'_{red}=pt\).
Then the stable SUSY curve obtained from \((M, \{s_i\}, \mathcal{R}, \mathcal{D})\) via base change along \(b\) has the same dual graph because the number of irreducible components, their intersection behavior and labelings of punctures are invariant under base change.
In particular, we can assume without loss of generality that we consider the  dual graph of a SUSY curve over \(B=pt\), that is equivalently, a spin curve.

\begin{rem}
	The dual graph of a SUSY curve of genus zero is a SUSY tree.
	More generally, the genus of a stable SUSY curve and its dual graph coincide.
\end{rem}

The Definition~\ref{defn:DualGraph} of dual graphs of SUSY curves is compatible with the definition of dual graphs of algebraic nodal curves.
Indeed, the definition of dual graphs of algebraic nodal curves coincides with the points~\ref{item:defnDualGraph:Vertices}--\ref{item:defnDualGraph:Genus} of Definition~\ref{defn:DualGraph} and has a simplified version of~\ref{item:defnDualGraph:Punctures}.
This yields:
\begin{prop}
	Let \(M\) be a stable SUSY curve of genus \(g\) with Neveu--Schwarz punctures labeled by \(I_{NS}\) and Ramond punctures labeled by \(I_R\).
	Its reduction \(M_{red}\) is a stable curve of genus \(g\) with marked points obtained from the reduction of Neveu--Schwarz and Ramond punctures and labeled by \(I_{NS}\cup I_R\).
	The dual graph \(\tau\in\cat{SGr}^{\cat{st}}_{g,I_{NS},I_R}\) of \(M\) and the dual graph \(\sigma\in\cat{MGr}^{\cat{st}}_{g,I_{NS}\cup I_R}\) of \(M_{red}\) satisfy \(\sigma=\cat{F}_{g,I_{NS},I_R}(\tau)\).
\end{prop}

\section{SUSY operads}\label{Sec:SupermodularOperads}
In this section, we give a method to construct SUSY operads from simple building blocks and use that method to construct the SUSY operad of moduli superstacks.
Recall from~\cite{BM-GOIC} that a generalized operad is a monoidal functor \(\cat{\Gamma}\to \cat{C}\) from a category of graphs \(\cat{\Gamma}\) to a \enquote{ground category} \(\cat{C}\) that maps graftings to isomorphisms.
All operads in this section are generalized operads in the sense of~\cite{BM-GOIC}.

The SUSY operad we will construct is a supergeometric generalization of the modular operad
\begin{equation}
	\begin{split}
		\cat{o}\colon \cat{MGr^{st}}&\to \cat{M} \\
		\sigma &\mapsto \prod_{v \in V_\sigma} \overline{M}_{g_\sigma(v), F_\sigma(v)}
	\end{split}
\end{equation}
that sends stable modular graphs to products of moduli stacks \(\overline{M}_{g_\sigma(v), F_\sigma(v)}\) of stable algebraic curves with genus \(g_\sigma(v)\) and marked points labeled by \(F_\sigma(v)\).
Central to the proof of functoriality of \(\cat{o}\) is that contractions of loops and edges are mapped to gluing maps
\begin{align}
	gl^l(i, i')\colon \overline{M}_{g, I}&\to \overline{M}_{g+1, I\setminus\Set{i,i'}}, &
	gl(i, i')\colon \overline{M}_{g, I}\times \overline{M}_{g', I'} &\to \overline{M}_{g+g', \left(I\cup I'\right)\setminus \Set{i,i'}}.
\end{align}
together with commutativity relations for successive gluings and relabelings.
The operad~\(\cat{o}\) is known to encode interesting data about the geometry and topology of the compact moduli stack \(\overline{M}_{g, I}\).
For further information about the modular operads, see~\cites{GK-MO}{M-FMQCMS}{CM-SGZMO}{CMM-QO}.

The first step in the construction of the SUSY operad \(\cat{O}\colon \cat{SGr^{st}}\to \cat{SM}\) generalizing \(\cat{o}\) is to show that SUSY operads can be constructed from similar building blocks.
That is, the operad \(\cat{O}\) is determined by its value on corollas, automorphisms of corollas and the contractions of loops and edges of either Neveu--Schwarz or Ramond type.
A second challenge is the construction of a suitable target category \(\cat{SM}\) as stable SUSY curves cannot be uniquely glued along Ramond punctures and morphisms between their moduli spaces do not necessarily project onto morphisms between moduli spaces of classical stable curves.

In Section~\ref{SSec:OperadsOnSUSYGraphs} we demonstrate that SUSY operads in any category can be constructed by specifying them on corollas and certain simple maps.
In Section~\ref{SSec:ModuliSpacesOfStableSUSYCurves} we recall the necessary background on moduli stacks of stable SUSY curves with punctures and gluings.
The SUSY operad \(\cat{O}\colon \cat{SGr^{st}}\to \cat{SM}\) and the target category \(\cat{SM}\) is then constructed in Section~\ref{SSec:SUSYOperadInModuliStacks}

\subsection{Operads on SUSY graphs}\label{SSec:OperadsOnSUSYGraphs}
In this section we prove a proposition that allows to construct operads \(\cat{O}\colon \cat{SGr^{st}}\to \cat{C}\) by specifying the image of corollas, isomorphisms of corollas and contractions of single edges together with certain compatibility conditions.
The resulting description is similar to the definition of modular operads in~\cite{GK-MO}.
The strategy is to break graphs into corollas and morphisms of graphs into automorphisms of corollas, contractions of single edges as well as graftings.

We recall the notions of total grafting and atomization from~\cite{BM-GOIC} that exist by assumption for every category of labeled graphs.
Given a SUSY graph \(\sigma\), the \emph{total grafting} is the unique morphism in \(\cat{SGr^{st}}\)
\begin{equation}\label{eq:TotalGrafting}
	\circ_{\sigma}\colon \coprod_{v\in V_\sigma} \sigma_v\to \sigma
\end{equation}
which is bijective on vertices as well as flags and grafts all edges.
The graphs \(\sigma_v\) are the corollas corresponding to the vertices \(v\) with their bounding flags.

Given a morphism \(h\colon \tau\to \sigma\) in the category \(\cat{SGr^{st}}\) the \emph{atomization} is a commutative diagram of the form
\begin{diag}\label{diag:Atomization}
	\matrix[mat](m){
		\coprod_{v\in V_\sigma} \tau_v & \coprod_{v\in V_\sigma} \sigma_v \\
		\tau & \sigma \\
	} ;
	\path[pf]{
		(m-1-1) edge node[auto]{\(\coprod h_v\)} (m-1-2)
			edge node[auto]{\(n\)} (m-2-1)
		(m-1-2) edge node[auto]{\(\circ_{\sigma}\)}(m-2-2)
		(m-2-1) edge node[auto]{\(h\)}(m-2-2)
	};
\end{diag}
in \(\cat{SGr^{st}}\) where
\begin{itemize}
	\item
		the graphs \(\tau_v\) are defined by
		\begin{align}
			V_{\tau_v} &= \Set{w\in V_\tau \given h_V(w)=v} &
			F_{\tau_v} &= \Set{f\in F_\tau \given h_V(\partial_\tau f)=v} \\
			\partial_{\tau_v} &= \partial_\tau|_{V_{\tau_v}} &
			j_{\tau_v} &= j_\tau|_{\tau_v}
		\end{align}
	\item
		the morphism \(h_v\colon \tau_v\to \sigma_v\) is given by
		\begin{align}
			h_{v,V} &= h_V|_{V_{\tau_v}} &
			h_v^F &= h^F|_{F_{\tau_v}} &
			j_{h_v} &= j_h|_{F_{\tau_v}}
		\end{align}
		and
	\item
		the morphism \(n\) is given by the grafting of the \(\tau_v\).
\end{itemize}
Both total grafting and atomization are unique up to unique isomorphism.

The maps \(h_v\colon \tau_v\to \sigma_v\) are maps to a corolla and can consequently be decomposed into
\begin{diag}\label{diag:DecompositionMapToCorolla}
	\matrix[seq](m){
		\tau_v & \tau_v^1 & \tau_v^2 & \tau_v^3 & \dots & \tau_v^{l_v} & \sigma_v\\
	} ;
	\path[pf]{
		(m-1-1) edge node[auto]{\(gr_{\tau_v}\)} (m-1-2)
			edge[bend right=30] node[auto]{\(h_v\)} (m-1-7)
		(m-1-2) edge node[auto]{\(h_v^1\)} (m-1-3)
		(m-1-3) edge node[auto]{\(h_v^2\)} (m-1-4)
		(m-1-4) edge (m-1-5)
		(m-1-5) edge (m-1-6)
		(m-1-6) edge node[auto]{\(h_v^{l_v}\)} (m-1-7)
	};
\end{diag}
where \(gr_{\tau_v}\) is a grafting and each \(h_v^l\) is a contraction of a single edge of \(\tau_v\) or an isomorphism.
The atomization of a contraction of a single edge consists of the identity maps for all but one vertex of the target.
The vertex which is not the image of the identity is either the image of a contraction of a single Neveu--Schwarz resp.\ Ramond edge connecting two vertices or the image of the contraction of a single Neveu--Schwarz resp.\ Ramond loop at a single vertex.
The atomization of an isomorphism is a grafting of isomorphisms of corollas.

Of course the decomposition~\eqref{diag:DecompositionMapToCorolla} is not unique but one can prove the following two Lemmata using the atomization:
\begin{lemma}\label{lemma:CommuteIsoContraction}
	Let \(\tau\) be a SUSY graph, and \(a\colon \tau\to \sigma\) an isomorphism and \((f, f')\) a pair of flags of \(\sigma\) that forms an edge.
	Denote by
	\begin{align}
		con_\sigma(f, f')\colon \sigma &\to \sigma_{(f, f')} &
		con_\tau(a^F(f), a^F(f'))\colon \tau &\to \tau_{(a^F(f), a^F(f'))}
	\end{align}
	the maps that contract the edge \((f, f')\) and \((a^F(f), a^F(f'))\) respectively.
	Then there exists a unique isomorphism \(a_{(f, f')}\colon \tau_{(a^F(f), a^F(f'))} \to \sigma_{(f, f')}\) such that
	\begin{equation}
		con_\sigma(f,f')\circ a = a_{(f, f')}\circ con_\tau(a^F(f), a^F(f')).
	\end{equation}
\end{lemma}
\begin{proof}
	The atomization of an isomorphism implies that it is obtained as a grafting of isomorphisms of corollas.
	The atomization of a contraction of a single edge implies that all except one of the maps \(h_v\) in the atomization diagram are identities.
	Hence it remains to consider the two special cases where either
	\begin{itemize}
		\item
			\(\sigma\) consists of a single vertex and \(con_\sigma(f,f')\) is the contraction of the only loop, or
		\item
			\(\sigma\) is a graph with two vertices and \(con_\sigma(f,f')\) is the contraction of the only edge between the two vertices.
	\end{itemize}
	In both cases one can construct the isomorphism \(a_{(f,f')}\) explicitly.
\end{proof}
\begin{lemma}\label{lemma:CommuteContractions}
	Let \(\tau\) be a SUSY graph and \((f_1, f'_1)\) and \((f_2,f'_2)\) two pairs of flags that form edges.
	Let \(con_\tau(f_i, f'_i)\colon \tau\to \tau_{(f_i, f'_i)}\), \(i=1, 2\) be the contraction of the edge \((f_i, f'_i)\) and
	\begin{align}
		con_{\tau_{(f_1, f'_1)}}(f_2, f'_2)\colon \tau_{(f_1, f'_1)}& \to \tau_{(f_1, f'_1);(f_2, f'_2)} &
		con_{\tau_{(f_2, f'_2)}}(f_1, f'_1)\colon \tau_{(f_2, f'_2)} &\to \tau_{(f_2, f'_2);(f_1, f'_1)}
	\end{align}
	the contraction of the respective other edge.
	Then \(\tau_{(f_1, f'_1);(f_2, f'_2)}=\tau_{(f_2, f'_2);(f_1, f'_1)}\) and
	\begin{equation}
		con_{\tau_{(f_1, f'_1)}}(f_2, f'_2)\circ con_\tau(f_1, f'_1) = con_{\tau_{(f_2, f'_2)}}(f_1, f'_1)\circ con_\tau(f_2, f'_2).
	\end{equation}
\end{lemma}
\begin{proof}
	Again, the atomization of a contraction of a single edge is an identity on all except one target vertex.
	If the two contractions do not share a vertex, they obviously commute.
	If they share a vertex, it remains to consider the following four cases:
	\begin{itemize}
		\item
			\(\tau\) consist of a single vertex and \(con_\tau(f_i,f'_i)\), \(i=1,2\) are contractions of the only two loops,
		\item
			\(\tau\) consists of two vertices and \(con_\tau(f_1, f'_1)\) is the contraction of the only loop and \(con_\tau(f_2, f'_2)\) the contraction of the only edge connecting the two vertices,
		\item
			\(\tau\) consists of two vertices and the \(con_\tau(f_i, f'_i)\) are contractions of the only two edges connecting the two vertices,
		\item
			\(\tau\) consist of three vertices and \(con_\tau(f_i, f'_i)\) are the contractions of the only two edges connecting the three vertices.
	\end{itemize}
	In all of the above cases one can verify explicitly that the contractions \(con_{(f_i, f'_i)}\) commute.
\end{proof}

Suppose now that we have an operad \(\cat{O}\colon \cat{SGr^{st}}\to \cat{C}\) in the sense of~\cite[Definition~1.6.1]{BM-GOIC}.
That is, \(\cat{O}\) is a symmetric monoidal functor that sends graftings to isomorphisms.
Hence, in particular, the total grafting morphism \(\circ_\sigma\) from Equation~\eqref{eq:TotalGrafting} is sent to an isomorphism for any graph \(\sigma\) and it follows that
\begin{equation}
	\cat{O}(\sigma)
	= \cat{O}(\circ_\sigma)\left(\cat{O}\left(\coprod_{v\in V_\sigma} \sigma_v\right)\right)
	= \cat{O}(\circ_\sigma)\left(\bigotimes_{v\in V}\cat{O}(\sigma_v)\right).
\end{equation}
This determines \(\cat{O}(\sigma)\) in terms of the operation of \(\cat{O}\) on the corollas \(\sigma_v\) up to the isomorphism \(\cat{O}(\circ_\sigma)\).
In the following we assume that \(\cat{O}\) acts on \(\circ_\sigma\) and any other grafting as the identity.

For any map, \(h\colon \tau\to \sigma\), the map \(\cat{O}(h)\colon \cat{O}(\tau)\to \cat{O}(\sigma)\) is then identical to \(\bigotimes_{v\in V_\sigma} \cat{O}(h_v)\) where \(h_v\colon \tau_v\to \sigma_v\) as in the atomization diagram~\eqref{diag:Atomization}.
By the decomposition of each \(h_v\) into a sequence of graftings, isomorphism and contractions of single edges, see~\eqref{diag:DecompositionMapToCorolla}, the value of \(\cat{O}\) on any map is hence completely determined by specifying its image on isomorphisms of corollas, contractions of single loops at single vertices and contractions of single edges between two vertices.
However, those images cannot be specified arbitrarily.
As the decomposition~\eqref{diag:DecompositionMapToCorolla} is not unique, the functor \(\cat{O}\) needs to respect the commutation relations explained in Lemma~\ref{lemma:CommuteIsoContraction} and Lemma~\ref{lemma:CommuteContractions}.
The following Proposition~\ref{prop:OperadConstruction}, spells out those commutation relations in detail.

\begin{prop}\label{prop:OperadConstruction}
	Let \(\cat{C}\) be a category with
	\begin{itemize}
		\item
			objects \(C(g, I_{NS}, I_R)\) for every non-negative integer \(g\in \Z_{\geq 0}\), and finite sets \(I_{NS}\) and \(I_R\) such that the cardinality \(\#I_R\) is even and \(2g - 2 + \#I_{NS} + \#I_R > 0\)
		\item
			for every pair \((s_{NS}, s_R)\) of isomorphisms of finite sets \(s_{NS}\colon I_{NS}\to J_{NS}\) and \(s_R\colon I_R\to J_R\) an isomorphism
			\begin{equation}
				c_{(s_{NS}, s_R)}\colon C(g, I_{NS}, I_R)\to C(g, J_{NS}, J_R),
			\end{equation}
		\item
			for any two distinct elements \(i, i'\in I_{NS}\) and distinct elements \(j, j'\in I_R\) there are maps
			\begin{align}
				c^l_{NS}(i, i')\colon C(g, I_{NS}, I_R) &\to C(g+1, I_{NS}\setminus\Set{i, i'}, I_R), \\
				c^l_R(j, j')\colon C(g, I_{NS}, I_R) &\to C(g+1, I_{NS}, I_R\setminus\Set{j, j'}),
			\end{align}
		\item
			for any pairs of elements \((i, i')\in I_{NS}\times I'_{NS}\) and \((j,j')\in I_R\times I'_R\) maps
			\begin{align}
				\begin{split}
					\MoveEqLeft
					c_{NS}(i, i')\colon C(g, I_{NS}, I_R)\times C(g, I'_{NS}, I'_R) \\
					&\to C\left(g+g, \left(I_{NS}\cup I'_{NS}\right)\setminus\Set{i, i'}, I_R\cup I'_R\right),
				\end{split} \\
				\begin{split}
					\MoveEqLeft
					c_R(j, j')\colon C(g, I_{NS}, I_R)\times C(g, I'_{NS}, I'_R) \\
					&\to C\left(g+g, I_{NS}\cup I'_{NS}, \left(I_R\cup I'_R\right)\setminus\Set{j, j'}\right).
				\end{split}
			\end{align}
	\end{itemize}
	There is a unique operad \(\cat{O}\colon \cat{SGr^{st}}\to \cat{C}\) such that
	\begin{itemize}
		\item
			for any SUSY corolla \(\tau\)
			\begin{equation}
				\cat{O}(\tau) = C(g_\tau, F_{\tau, NS}, F_{\tau, R}),
			\end{equation}
		\item
			for any isomorphism \(h\colon \tau\to \sigma\) of SUSY corollas
			\begin{equation}
				\cat{O}(h) = c_{\left({\left(h^F|_{F_{\sigma, NS}}\right)}^{-1}, {\left(h^F|_{F_{\sigma, R}}\right)}^{-1}\right)},
			\end{equation}
		\item
			for any grafting \(h\colon \tau\to \sigma\), that is \(h_V\) and \(h^F\) are identities, \(\cat{O}(h)\) is the identity.
		\item
			for any contraction \(con^l_{(f, f')}\colon \tau\to \tau\) of the loop specified by \(f, f'\in F_{\tau}\) of the same color \(c_\tau(f)\in \Set{NS, R}\) at a SUSY graph~\(\tau\) with a single vertex
			\begin{equation}
				\cat{O}\left(con^l_{(f, f')}\right)= c^l_{c_\tau(f)}(f, f'),
			\end{equation}
		\item
			for any contraction \(con_{(f, f')}\colon \tau\to\sigma\) of the edge in a graph \(\tau\) consisting of two SUSY corollas \(\tau_1\), \(\tau_2\) grafted at the edge specified by \((f, f')\in F_{\tau_1}\times F_{\tau_2}\) of the same color \(c_\tau(f)\in \Set{NS, R}\)
			\begin{equation}
				\cat{O}\left(con_{(f, f')}\right)= c_{c_\tau(f)}(f, f'),
			\end{equation}
	\end{itemize}
	if and only if the following conditions are satisfied:
	\begin{enumerate}
		\item\label{item:OperadConstruction:IsoGroup}
			The isomorphisms \(c_{(s_{NS}, s_R)}\) respect the composition of isomorphisms of finite sets.
			That is, for two pairs of isomorphisms of finite sets
			\begin{align}
				(s_{NS}\colon I_{NS}\to J_{NS} &, s_R\colon I_R\to J_R) &
				(s'_{NS}\colon J_{NS}\to K_{NS} &, s'_R\colon J_R\to K_R)
			\end{align}
			it holds
			\begin{equation}
				c_{\left(s'_{NS}\circ s_{NS}, s'_R\circ s_R\right)}=c_{\left(s'_{NS}, s'_R\right)}\circ c_{\left(s_{NS}, s_R\right)}.
			\end{equation}
		\item\label{item:OperadConstruction:CommutationIsoLoop}
			The isomorphisms \(c_{(s_{NS}, s_R)}\) \enquote{commute} with the maps \(c^l_{NS}(i, i')\) and \(c^l_R(j, j')\).
			That is, for any pair \((s_{NS}\colon I_{NS}\to J_{NS}, s_R\colon I_R\to J_R)\) of isomorphisms of finite sets, elements \(i, i'\in I_{NS}\) and \(j, j'\in I_R\)
			\begin{align}
				c_{(\overline{s}_{NS}, s_R)}\circ c^l_{NS}(i, i') &= c^l_{NS}(s_{NS}(i), s_{NS}(i'))\circ c_{(s_{NS}, s_R)} \\
				c_{(s_{NS}, \overline{s}_R)}\circ c^l_R(j, j') &= c^l_R(s_R(j), s_R(j'))\circ c_{(s_{NS}, s_R)}
			\end{align}
			Here
			\begin{align}
				\overline{s}_{NS}\colon I_{NS}\setminus \Set{i,i'} &\to J_{NS}\setminus\Set{s_{NS}(i), s_{NS}(i')} \\
				\overline{s}_R\colon I_R\setminus \Set{i,i'} &\to J_R\setminus\Set{s_R(i), s_R(i')}
			\end{align}
			are the induced restricted isomorphisms of finite sets.
		\item\label{item:OperadConstruction:CommutationIsoEdge}
			The isomorphisms \(c_{(s_{NS}, s_R)}\) \enquote{commute} with the maps \(c_{NS}(i, i')\) and \(c_R(j, j')\).
			That is, for two pairs of isomorphisms of finite sets
			\begin{align}
				(s_{NS}\colon I_{NS}\to J_{NS} &, s_R\colon I_R\to J_R) &
				(s'_{NS}\colon I'_{NS}\to J'_{NS} &, s'_R\colon I'_R\to J'_R)
			\end{align}
			it holds
			\begin{align}
				c_{NS}\left(s_{NS}(i), s'_{NS}(i')\right) \circ \left(c_{(s_{NS}, s_R)}\times c_{(s'_{NS}, s'_R)}\right)
				&= c_{\left(\overline{s_{NS}\cup s'_{NS}}, s_R\cup s'_R\right)} \circ c_{NS}(i, i'), \\
				c_R\left(s_R(j), s'_R(j')\right) \circ \left(c_{\left(s_{NS}, s_R\right)}\times c_{(s'_{NS}, s'_R)}\right)
				&= c_{\left(s_{NS}\cup s'_{NS}, \overline{s_R\cup s'_R}\right)} \circ c_R(i, i').
			\end{align}
			Here
			\begin{align}
				s_{NS}\cup s'_{NS}\colon I_{NS}\cup I'_{NS} &\to J_{NS}\cup J'_{NS} \\
				\overline{s_{NS}\cup s'_{NS}}\colon \left(I_{NS}\cup I'_{NS}\right)\setminus \Set{i,i'} &\to \left(J_{NS}\cup J'_{NS}\right)\setminus\Set{s_{NS}(i), s'_{NS}(i')} \\
				s_R\cup s'_R\colon I_R\cup I'_R &\to J_R\cup J'_R \\
				\overline{s_R\cup s'_R}\colon \left(I_R\cup I'_R\right)\setminus \Set{i,i'} &\to \left(J_R\cup J_R\right)\setminus\Set{s_R(i), s'_R(i')}
			\end{align}
			are the induced isomorphisms of finite sets.
		\item\label{item:OperadConstruction:CommutationLoopLoop}
			The maps \(c^l_{NS}(i, i')\) and \(c^l_R(j, j')\) \enquote{commute} with each other.
			That is for any distinct elements \(i_1, i'_1, i_2, i'_2\in I_{NS}\) and distinct \(j_1, j'_1, j_2, j'_2\in I_R\) it holds
			\begin{align}
				c^l_{NS}(i_1,i'_1)\circ c^l_{NS}(i_2, i'_2) &= c^l_{NS}(i_2, i'_2) \circ c^l_{NS}(i_1, i'_1) \\
				c^l_{NS}(i_1, i'_1)\circ c^l_R(j_1, j'_1) &= c^l_R(j_1, j'_1) \circ c^l_{NS}(i_1, i'_1) \\
				c^l_R(j_1, j'_1)\circ c^l_R(j_2, j'_2) &= c^l_R(j_2, j'_2) \circ c^l_R(j_1, j'_1)
			\end{align}
		\item\label{item:OperadConstruction:CommutationLoopEdge}
			The maps \(c^l_{NS}(i,i')\) and \(c^l_R(j,j')\) \enquote{commute} with \(c_{NS}(i,i')\) and \(c_R(j,j')\).
			That is, for any distinct \(i_1, i'_1, i_2\in I_{NS}\), \(i'_2\in I'_{NS}\) and \(j_1, j'_1,j_2\in I_R\), \(j'_2\in I'_R\) we have
			\begin{align}
				c^l_{NS}(i_1, i'_1)\circ c_{NS}(i_2, i'_2) &= c_{NS}(i_2, i'_2)\circ \left(c^l_{NS}(i_1, i'_1) \times \id_{C(g', I'_{NS}, I'_R)}\right) \\
				c^l_R(j_1, j'_1)\circ c_{NS}(i_2, i'_2) &= c_{NS}(i_2, i'_2)\circ \left(c^l_R(j_1, j'_1) \times \id_{C(g', I'_{NS}, I'_R)}\right) \\
				c^l_{NS}(i_1, i'_1)\circ c_R(j_2, j'_2) &= c_R(j_2, j'_2)\circ \left(c^l_{NS}(i_1, i'_1) \times \id_{C(g', I'_{NS}, I'_R)}\right) \\
				c^l_R(j_1, j'_1)\circ c_R(j_2, j'_2) &= c_R(j_2, j'_2)\circ \left(c^l_R(j_1, j'_1) \times \id_{C(g', I'_{NS}, I'_R)}\right)
			\end{align}
			Furthermore, for any distinct pairs \((i_1, i'_1), (i_2, i'_2) \in I_{NS}\times I'_{NS}\) and distinct pairs \((j_1, j'_1), (j_2, j'_2) \in I_R\times I'_R\)
			\begin{align}
				c^l_{NS}(i_2, i'_2)\circ c_{NS}(i_1, i'_1) &= c^l_{NS}(i_1, i'_1)\circ c_{NS}(i_2, i'_2) \\
				c^l_R(j_2, j'_2)\circ c_{NS}(i_1, i'_1) &= c^l_{NS}(i_1, i'_1)\circ c_R(j_2, j'_2) \\
				c^l_R(j_2, j'_2)\circ c_R(j_1, j'_1) &= c^l_R(j_1, j'_1)\circ c_R(j_2, j'_2)
			\end{align}
		\item\label{item:OperadConstruction:CommutationEdgeEdge}
			The maps \(c_{NS}(i, i')\) and \(c_R(j, j')\) \enquote{commute} with each other.
			That is for any pairs of elements \((i_1, i'_1)\in I_{NS}\times I'_{NS}\), \((i_2, i'_2)\in I'_{NS}\times I''_{NS}\) with \(i'_1\neq i_2\) and \((j_1, j'_1)\in I_R\times I'_R\), \((j_2, j'_2)\in I'_R\times I''_R\) with \(j'_1\neq j_2\) it holds
			\begin{align}
				\begin{split}
					\MoveEqLeft
					c_{NS}(i_1, i'_1)\circ \left(\id_{C(g, I_{NS}, I_R)} \times c_{NS}(i_2, i'_2)\right) \\
					&= c_{NS}(i_2, i'_2)\circ \left(c_{NS}(i_1, i'_1) \times \id_{C(g'', I''_{NS}, I''_R)}\right)
				\end{split} \\
				\begin{split}
					\MoveEqLeft
					c_{NS}(i_1, i'_1)\circ \left(\id_{C(g, I_{NS}, I_R)} \times c_R(j_2, j'_2)\right) \\
					&= c_R(j_2, j'_2)\circ \left(c_{NS}(i_1, i'_1) \times \id_{C(g'', I''_{NS}, I''_R)}\right)
				\end{split} \\
				\begin{split}
					\MoveEqLeft
						c_R(j_1, j'_1)\circ \left(\id_{C(g, I_{NS}, I_R)} \times c_R(j_2, j'_2)\right) \\
						&= c_R(j_2, j'_2)\circ \left(c_R(j_1, j'_1) \times \id_{C(g'', I''_{NS}, I''_R)}\right)
				\end{split}
			\end{align}
	\end{enumerate}
\end{prop}
\begin{proof}
	We have spelled out the construction of \(\cat{O}\) on graphs and graph morphisms before Proposition~\ref{prop:OperadConstruction}.
	The construction is independent of the decomposition of graph morphisms because the commutativity conditions~\ref{item:OperadConstruction:CommutationIsoLoop}--\ref{item:OperadConstruction:CommutationEdgeEdge} imply that Lemma~\ref{lemma:CommuteIsoContraction} and Lemma~\ref{lemma:CommuteContractions} also hold in \(\cat{C}\).
	Here, the condition~\ref{item:OperadConstruction:CommutationIsoLoop} corresponds to the first special case in the proof of Lemma~\ref{lemma:CommuteIsoContraction}, and condition~\ref{item:OperadConstruction:CommutationIsoEdge} to the second.
	The conditions~\ref{item:OperadConstruction:CommutationLoopLoop}--\ref{item:OperadConstruction:CommutationEdgeEdge} correspond to the four special cases in the proof of Lemma~\ref{lemma:CommuteContractions} in the same order.

	It remains to show that \(\cat{O}\) is a functor, that is for two graph morphisms \(h\colon \tau\to \sigma\) and \(f\colon \sigma \to \rho\) we have \(\cat{O}(f\circ h) = \cat{O}(f)\circ \cat{O}(h)\).
	To this end note that the atomization of the composition \(f\circ h\) can be obtained from the atomizations of \(f\) and \(h\) via graftings.
\end{proof}

Note that for any SUSY operad \(\cat{O}\colon \cat{SGr^{st}}\to \cat{C}\) there is an induced modular operad \(\cat{\tilde{O}}=\cat{O}\circ\cat{I}\colon \cat{MGr^{st}}\to \cat{SGr^{st}}\) where \(\cat{I}\colon \cat{MGr^{st}}\to \cat{SGr^{st}}\) is the inclusion functor.
This operad \(\cat{\tilde{O}}\) can be specified by less data:
\begin{cor}\label{cor:ConstructionModularOperad}
	Let \(\cat{C}\) be a category with objects \(C(g, I, \emptyset)\), and morphisms \(c_{(s, \emptyset)}\), \(c^l_{NS}(i,i')\) and \(c_{NS}(i,i')\) satisfying the conditions~\ref{item:OperadConstruction:IsoGroup}--\ref{item:OperadConstruction:CommutationEdgeEdge} of Proposition~\ref{prop:OperadConstruction} for the special case \(I_R=\emptyset\).
	There is a modular operad \(\cat{\tilde{O}}\colon \cat{MGr^{st}}\to \cat{C}\) such that corollas are mapped to \(C(g, I, \emptyset)\), graftings to identities, automorphisms of corollas to \(c_{(s, \emptyset)}\) and contractions of single loops at a vertex to \(c^l(i,i')\) and contractions of a single edge between two vertices to \(c_{NS}(i,i')\).
\end{cor}

\begin{rem}\label{rmk:kLabeling}
	Even though the objects \(C(g, I_{NS}, I_R)\) and \(C(g, I'_{NS}, I'_R)\) are isomorphic if \(\#I_{NS}=\#I'_{NS}\) and \(\#I_R=\#I'_R\) it is still important to label them with finite sets instead of their cardinality because the isomorphisms are not canonical.
	Otherwise, the image under \(\cat{O}\) of an isomorphism \(\tau\to \sigma\) between two corollas both having genus \(g\), \(k_{NS}\) Neveu--Schwarz tails and \(k_R\) Ramond tails would be an automorphism of \(C(g, k_{NS}, k_R)\).
	In general, however, there is no canonical isomorphism \(\tau\to \sigma\) which could be mapped to the identity of \(C(g, k_{NS}, k_R)\).

	If one insists on labeling the images of corollas by integers such as \(C(g, k_{NS}, k_R)\), one can choose a skeleton \(\cat{Sk SGr^{st}}\) of the category of stable SUSY graphs together with an equivalence of categories between \(\cat{Sk SGr^{st}}\) and \(\cat{SGr^{st}}\) and construct the operad as a functor \(\cat{Sk SGr^{st}}\to \cat{C}\).
	The choice of skeleton and equivalence avoids the above ambiguity because it consistently selects representatives of isomorphism classes of objects and morphisms.
\end{rem}

\subsection{Moduli spaces of stable SUSY curves}\label{SSec:ModuliSpacesOfStableSUSYCurves}
In~\cite{FKP-MSSCCLB} it was proved (Theorem A), that for each pair of finite sets $(I_{NS}, I_R)$ with \(\#I_R\) even, the functor of families of stable SUSY curves of genus $g$ with punctures labeled by \(I_{NS}\) and \(I_R\) respectively is represented by a smooth and proper Deligne--Mumford superstack $\overline{\mathcal{M}}_{g, I_{NS}, I_R}$.
The reader can find a categorical background of stacks in~\cite{O-ASS}.
Basic geometric definitions of the theory of superstacks are presented in~\cite{BR-SMSCRP}.

The reduced space \({\left(\overline{\mathcal{M}}_{g,I_{NS},I_R}\right)}_{red}\) is the moduli stack of twisted spin curves.
Here a twisted spin curve is a pair \(({M}_{red},S)\) consisting of an algebraic stable curve \(M_{red}\) together with a twisted spinor bundle satisfying~\eqref{eq:TwistedSpinorBundle}, see, for example,~\cite{AJ-MTSC}.
We denote the moduli stack of pairs \((M_{red}, S)\) by \(\overline{M}^{spin}_{g, I_{NS},I_R}={\left(\overline{\mathcal{M}}_{g,I_{NS},I_R}\right)}_{red}\).

The moduli space \(\overline{M}^{spin}_{g,I_{NS},I_R}\) is a bundle over the moduli space \(\overline{M}_{g, I_{NS} \cup I_R}\) of algebraic curves with marked points labeled by \(I_{NS}\cup I_R\) where the map
\begin{equation}
	\pi\colon \overline{M}^{spin}_{g, I_{NS},I_R}\to \overline{M}_{g, I_{NS} \cup I_R}
\end{equation}
forgets the spinor bundle, that is, sends \((M_{red}, S)\) to \(M_{red}\).
In the case of genus zero the map \(\pi\) is an isomorphism for all pairs \((I_{NS}, I_R=\emptyset)\).

Let \(s=(s_{NS}\colon I_{NS}\to J_{NS}, s_R\colon I_R\to J_R)\) be a pair of isomorphisms of finite sets and \((M, \Set{s_i}, \mathcal{R}, \delta)\) be a stable SUSY curve with Neveu--Schwarz punctures labeled by \(I_{NS}\) and Ramond punctures labeled by \(I_R\).
We denote by \(s(M, \{s_i\}, \mathcal{R}, \delta)\) the stable SUSY curve obtained from \((M, \{s_i\}, \mathcal{R}, \delta)\) by relabeling the Neveu--Schwarz punctures by \(s_{NS}\) and the Ramond divisors by \(s_R\).
This renumbering of the punctures yields an isomorphism of moduli stacks~\(a_s\colon \overline{\mathcal{M}}_{g, I_{NS}, I_R} \to \overline{\mathcal{M}}_{g, J_{NS}, J_R}\).

The pair of isomorphisms \(s\) also yields an isomorphism \(a_{s,red}\colon \overline{M}^{spin}_{g,I_{NS}, I_R}\to \overline{M}^{spin}_{g,J_{NS}, J_R}\), as well as an automorphism \(a_{s_{NS}\cup s_R}\colon \overline{M}_{g,I_{NS}\cup I_R} \to \overline{M}_{g,J_{NS}\cup J_R}\) when considering the induced isomorphism \(s_{NS}\cup s_R\colon I_{NS}\cup I_R\to J_{NS}\cup J_R\).
Those isomorphisms fit into a commutative diagram
\begin{diag}
	\matrix[mat, column sep=huge](m){
		\overline{\mathcal{M}}_{g,I_{NS}, I_R} & \overline{\mathcal{M}}_{g,J_{NS}, J_R} \\
		\overline{M}^{spin}_{g,I_{NS}, I_R} & \overline{M}^{spin}_{g,J_{NS}, J_R} \\
		\overline{M}_{g,I_{NS}\cup I_R} & \overline{M}_{g,J_{NS}\cup J_R} \\
	} ;
	\path[pf]{
		(m-1-1) edge node[auto]{\(a_s\)} (m-1-2)
		(m-2-1) edge node[auto]{\(a_{s,red}\)} (m-2-2)
			edge node[auto]{\(i_{red}\)} (m-1-1)
			edge node[auto]{\(\pi\)} (m-3-1)
		(m-2-2) edge node[auto]{\(i_{red}\)} (m-1-2)
			edge node[auto]{\(\pi\)} (m-3-2)
			(m-3-1) edge node[auto]{\(a_{s_{NS}\cup s_R}\)} (m-3-2)
	};
\end{diag}
If \(J_{NS}=I_{NS}\) and \(J_R=I_R\) the isomorphism \(a_s\), \(a_{s,red}\) and \(a_{s_{NS}\cup s_R}\) are automorphisms of the moduli stacks.
If \(\#I\geq 3\) all automorphisms of \(\overline{M}_{0,I}\) are obtained from automorphisms \(I\to I\), and if \(\#I\geq 5\) the automorphisms are in bijection with elements of the automorphism group of \(I\), see~\cites{BM-TAGM0n}{M-TAGMgn}.

The stack \(\overline{\mathcal{M}}_{g, I_{NS}, I_R}\) is endowed with a \emph{boundary Cartier divisor} $\Delta = \Delta_{NS} + \Delta_R$, with normal crossings, see Section~8 of~\cite{FKP-MSSCCLB}.
The boundary divisor \(\Delta_{NS}\) encodes stable SUSY curves with at least one Neveu--Schwarz node and the boundary divisor \(\Delta_R\) encodes stable SUSY curves with at least one Ramond node.

Stable SUSY curves can be glued along punctures of the same type, as was worked out in~\cite[Section~8]{FKP-MSSCCLB}.
It is possible to either glue two punctures of the same type of a single SUSY curve or to glue together two separate SUSY curves along punctures of the same type.
While two Neveu--Schwarz punctures can be uniquely glued, the gluing of Ramond punctures has a degree of freedom.
We will now discuss the induced gluing maps on the moduli stacks of SUSY curves in the different cases.

Let \((M, \Set{s_i}, \mathcal{R}, \delta)\) be a SUSY curve of genus \(g\) with Neveu--Schwarz punctures labeled by \(I_{NS}\) and Ramond punctures by \(I_R\) and \((M', \Set{s'_i}, \mathcal{R'}, \delta')\) a SUSY curve of genus \(g'\) with Neveu--Schwarz punctures labeled by \(I'_{NS}\) and Ramond punctures labeled by \(I'_R\).
For any pair \((i, i')\in I_{NS}\times I'_{NS}\) we can glue \(M\) and \(M'\) along the Neveu--Schwarz punctures labeled by \(i\) and \(i'\) respectively to obtain a SUSY curve of genus \(g+g'\) with Neveu--Schwarz punctures labeled by \(\left(I_{NS}\cup I'_{NS}\right)\setminus \Set{i,i'}\) and Ramond punctures labeled by \(I_R\cup I'_R\).
On the level of moduli stacks this gluing yields an embedding of codimension~\(1|0\), see~\cite[Lemma~8.10]{FKP-MSSCCLB}:
\begin{equation}\label{eq:NSGluingSeparateComponents}
gl_{NS}(i,i')\colon \overline{\mathcal{M}}_{g, I_{NS}, I_{R}}\times\overline{\mathcal{M}}_{g', I'_{NS}, I'_{R}} \to \overline{\mathcal{M}}_{g+g', \left(I_{NS}\cup I'_{NS}\right)\setminus \Set{i,i'}, I_{R}\cup I'_R}
\end{equation}
This gluing map is compatible with the gluing on the level of spin moduli spaces and classical moduli spaces.
That is, the following diagram commutes:
\begin{diag}
	\matrix[mat, column sep=huge](m){
		\overline{\mathcal{M}}_{g, I_{NS}, I_{R}}\times\overline{\mathcal{M}}_{g', I'_{NS}, I'_{R}} & \overline{\mathcal{M}}_{g+g', \left(I_{NS}\cup I'_{NS}\right)\setminus \Set{i,i'}, I_{R}\cup I'_R} \\
		\overline{M}^{spin}_{g, I_{NS}, I_{R}}\times\overline{M}^{spin}_{g', I'_{NS}, I'_{R}} & \overline{M}^{spin}_{g+g', \left(I_{NS}\cup I'_{NS}\right)\setminus \Set{i,i'}, I_{R}+I'_R} \\
		\overline{M}_{g, I_{NS} \cup I_{R}}\times\overline{M}_{g', I'_{NS} \cup I'_{R}} & \overline{M}_{g+g', \left(I_{NS}\cup I'_{NS} \cup I_{R} \cup I'_R\right)\setminus \Set{i,i'}} \\
	} ;
	\path[pf]{
		(m-1-1) edge node[auto]{\(gl_{NS}(i,i')\)} (m-1-2)
		(m-2-1) edge node[auto]{\({gl_{NS}(i,i')}_{red}\)} (m-2-2)
			edge node[auto]{\(i_{red}\)} (m-1-1)
			edge node[auto]{\(\pi\)} (m-3-1)
		(m-2-2) edge node[auto]{\(i_{red}\)} (m-1-2)
			edge node[auto]{\(\pi\)} (m-3-2)
		(m-3-1) edge node[auto]{\(gl(i,i')\)} (m-3-2)
	};
\end{diag}

Gluing the two Neveu--Schwarz punctures with label \(i\) and \(i'\) of a single stable SUSY curve of genus \(g\) with Neveu--Schwarz punctures labeled by \(I_{NS}\) and Ramond punctures labeled by \(I_R\) yields a stable SUSY curve of genus \(g+1\) and Neveu--Schwarz punctures labeled by \(I_{NS}\setminus \Set{i,i'}\) and Ramond punctures labeled by \(I_R\).
On the level of moduli space this gluing yields an embedding \(gl^l_{NS}(i,i')\) which makes the following diagram commutative:
\begin{diag}
	\matrix[mat, column sep=huge](m){
		\overline{\mathcal{M}}_{g, I_{NS}, I_{R}} & \overline{\mathcal{M}}_{g+1, I_{NS}\setminus\Set{i,i'}, I_{R}} \\
		\overline{M}^{spin}_{g, I_{NS}, I_{R}} & \overline{M}^{spin}_{g+1, I_{NS}\setminus\Set{i,i'}, I_{R}} \\
		\overline{M}_{g, I_{NS} \cup I_{R}} & \overline{M}_{g+1, \left(I_{NS} \cup I_{R}\right)\setminus\Set{i,i'}} \\
	} ;
	\path[pf]{
		(m-1-1) edge node[auto]{\(gl^l_{NS}(i,i')\)} (m-1-2)
		(m-2-1) edge node[auto]{\({gl^l_{NS}(i,i')}_{red}\)} (m-2-2)
			edge node[auto]{\(i_{red}\)} (m-1-1)
			edge node[auto]{\(\pi\)} (m-3-1)
		(m-2-2) edge node[auto]{\(i_{red}\)} (m-1-2)
			edge node[auto]{\(\pi\)} (m-3-2)
		(m-3-1) edge node[auto]{\(gl^l(i,i')\)} (m-3-2)
	};
\end{diag}

Similar gluing maps for Ramond punctures are not defined uniquely because Ramond punctures are of dimension \(0|1\).
According to~\cite[Lemma~8.11]{FKP-MSSCCLB}, there is a principal fiber bundle \(\mathcal{P}_j\to\overline{\mathcal{M}}_{g, I_{NS}, I_R}\) that parametrize preferred coordinate systems of the \(j\)-th Ramond puncture.
The structure group of \(\mathcal{P}_j\) is of the form \(\Z_2\times \C^{0|1}\).
It follows from~\cite[Lemma~8.13]{FKP-MSSCCLB} that in order to glue two Ramond punctures \(j\) and \(j'\) one needs to choose an isomorphism between \(\mathcal{P}_j\) and \({[\ic]}_*\mathcal{P}_{j'}\), where \({[\ic]}_*\mathcal{P}_{j'}\) is the rescaling of the fibers of \(\mathcal{P}_{j'}\) by the complex unity \(\ic\).
Hence there are bundles of isometries and gluing maps
\begin{diag}\label{eq:RGluingSingleComponent}
	\matrix[mat, column sep=huge](m){
		\Iso(\mathcal{P}_j, {[\ic]}_*\mathcal{P}_{j'}) & \overline{\mathcal{M}}_{g + 1, I_{NS}, I_{R}\setminus\Set{j,j'}} \\
		\overline{\mathcal{M}}_{g, I_{NS}, I_R} & \\
	} ;
	\path[pf]{
		(m-1-1) edge node[auto]{\(gl^l_R(j,j')\)} (m-1-2)
			edge (m-2-1)
	};
\end{diag}
\begin{diag}\label{eq:RGluingSeparateComponents}
	\matrix[mat](m){
		\Iso(\mathcal{P}_j, {[\ic]}_*\mathcal{P}_{j'}) & \overline{\mathcal{M}}_{g + g',I_{NS} \cup I'_{NS}, \left(I_R \cup I'_R\right)\setminus\Set{j,j'}} \\
		\overline{\mathcal{M}}_{g, I_{NS}, I_R}\times \overline{\mathcal{M}}_{g', I'_{NS}, I'_R} &\\
	} ;
	\path[pf]{
		(m-1-1) edge node[auto]{\(gl_R(j,j')\)} (m-1-2)
			edge (m-2-1)
	};
\end{diag}
In both cases the gluing maps are embeddings of codimension \(1|0\), see~\cite[Lemma~8.14]{FKP-MSSCCLB}.

Also the gluing of Ramond punctures is compatible with the gluing of marked points on the moduli space of stable curves:
\begin{diag}\matrix[mat](m){
		&\Iso(\mathcal{P}_j, {[\ic]}_*\mathcal{P}_{j'}) & \\
		\overline{\mathcal{M}}_{g, I_{NS}, I_R} & & \overline{\mathcal{M}}_{g + 1, I_{NS}, I_{R}\setminus\Set{j,j'}} \\
		& {\Iso(\mathcal{P}_j, {[\ic]}_*\mathcal{P}_{j'})}_{red} & \\
		\overline{M}^{spin}_{g, I_{NS}, I_R} & & \overline{M}^{spin}_{g + 1, I_{NS}, I_{R}\setminus\Set{j,j'}} \\
		\overline{M}_{g, I_{NS}\cup I_R} & & \overline{M}_{g + 1, \left(I_{NS} \cup I_{R}\right)\setminus\Set{j,j'}} \\
	} ;
	\path[pf]{
		(m-1-2) edge node[auto]{\(gl^l_R(j,j')\)} (m-2-3)
			edge (m-2-1)
		(m-3-2) edge node[auto]{\({gl^l_R(j,j')}_{red}\)} (m-4-3)
			edge (m-4-1)
			edge node[auto]{\(i_{red}\)} (m-1-2)
		(m-4-1) edge node[auto]{\(i_{red}\)} (m-2-1)
			edge node[auto]{\(\pi\)} (m-5-1)
		(m-4-3) edge node[auto]{\(i_{red}\)} (m-2-3)
			edge node[auto]{\(\pi\)} (m-5-3)
		(m-5-1) edge node[auto]{\(gl^l(j,j')\)} (m-5-3)
	};
\end{diag}
Here, the bundle \({\Iso(\mathcal{P}_j, {[\ic]}_*\mathcal{P}_{j'})}_{red}\) is the \(\Z_2\)-bundle classifying the sign-choice in identifying the fibers of the spinor bundle of the pair \((M_{red}, S)\) at the two Ramond nodes with labels \(j\) and \(j'\).
A similar diagram exists for the gluing map \(gl_R(j,j')\).

One can work out commutativity relations between the isomorphism \(a_s\) induced by relabelings and the different gluing maps.
Likewise, as gluing is an operation local to the two punctures, there are various commutativity relations between the gluing maps.
One finds that they are analogous to the relations stated in the conditions~\ref{item:OperadConstruction:CommutationIsoLoop}--\ref{item:OperadConstruction:CommutationEdgeEdge} of Proposition~\ref{prop:OperadConstruction}.

For further information about moduli superspaces, see~\cites{DW-SMNP}{BR-SMSCRP}{CV-MPSC}{KSY-SQCI}.

\subsection{SUSY Operad in moduli superstacks}\label{SSec:SUSYOperadInModuliStacks}
In this section we will construct the operad \(\cat{O}\colon \cat{SGr^{st}}\to \cat{SM}\) of moduli superstacks of SUSY curves that extends the modular operad \(\cat{o}\colon \cat{MGr^{st}}\to \cat{M}\) of moduli stacks of algebraic curves.

Recall the following construction of the modular operad \(\cat{o}\):
Let \(\cat{M}\) be the full subcategory of the category of Deligne--Mumford stacks whose objects are finite products of moduli stacks \(\overline{M}_{g, I}\).
The operad \(\cat{o}\) sends modular corollas \(\tau\) to \(\cat{o}(\tau) = \overline{M}_{g, F_\tau}\), isomorphisms \(h\colon \tau\to \sigma\) of modular corollas to \(a_{{\left(h^F\right)}^{-1}}\), contractions of a single loop at one vertex to \(gl^l(i,i')\) and contraction of a single edge between two vertices to \(gl(i,i')\).
By Corollary~\ref{cor:ConstructionModularOperad}, this yields a well defined operad \(\cat{o}\colon \cat{MGr^{st}}\to \cat{M}\).

In this section we prove the following theorem:
\Operad{}

The main difficulty in the construction of the operad \(\cat{O}\colon \cat{SGr^{st}}\to \cat{SM}\) is to construct an appropriate symmetric monoidal category \(\cat{SM}\) such that Proposition~\ref{prop:OperadConstruction} can be applied.
The category \(\cat{SM}\) needs to contain products of moduli superstacks, their isomorphisms induced by relabelings and gluing maps.
In particular, we need to construct \(\cat{SM}\) containing the gluing maps \(gl_R(j,j')\) and \(gl_R^l(j,j')\) which are not morphisms between superstacks in the usual sense.
On the other hand, we have to choose the morphisms in \(\cat{SM}\) in such a way that the functor \(\cat{P}\) can be constructed as \(\cat{SM}\to \cat{SM_{red}}\to \cat{M}\) and \(\cat{O}\) extends the modular operad \(\cat{o}\).
We will construct the category in several steps.

\begin{defn}
	Let \(\cat{C_1}\) be the full subcategory of the category of Deligne--Mumford superstacks whose objects are finite products of compact moduli stacks \(\overline{\mathcal{M}}_{g,I_{NS}, I_R}\).
\end{defn}
The category \(\cat{C_1}\) is a symmetric monoidal category with respect to the product \(\prod\) of stacks.
Using Corollary~\ref{cor:ConstructionModularOperad}, one can construct a modular operad in~\(\cat{C_1}\):
\begin{prop}\label{prop:Operad1}
	The map
	\begin{equation}
		\tau \mapsto \prod_{v\in V_\tau} \overline{\mathcal{M}}_{g(v), F_\tau(v), \emptyset}
	\end{equation}
	can be extended to a modular operad \(\cat{O_1}\colon \cat{MGr^{st}}\to \cat{C_1}\) such that graftings are mapped to the identity, automorphism of corollas to the induced automorphisms of moduli superstacks and contractions to gluings of Neveu--Schwarz punctures.
\end{prop}

In order to extend the operad \(\cat{O_1}\) to a SUSY operad, we need to define a category \(\cat{C_2}\) such that the gluings of Ramond nodes are morphisms in this category:
\begin{defn}\label{defn:WeakMaps}
	Let \(M\) and \(M'\) be objects of \(\cat{C_1}\).
	We say that \(f\colon P\to M'\) is a weak map from \(M\) to \(M'\) if \(P\) is a fiber bundle over \(M\).
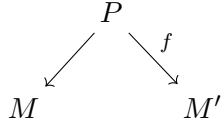
	\begin{diag}
		\matrix[dr](m){
			& P &\\
			M & & M'\\
		} ;
		\path[pf]{
			(m-1-2) edge (m-2-1)
				edge node[auto]{\(f\)} (m-2-3)
		};
	\end{diag}
	Weak maps can be composed in an obvious way that is also associative.

	Let \(\cat{C_2}\) be the category whose objects are objects of \(\cat{C_1}\), that is products of moduli stacks of punctured SUSY curves.
	Morphisms in \(\cat{C_2}\) are weak maps in the above sense.
\end{defn}
One can verify that the category \(\cat{C_2}\) is a symmetric monoidal category with respect to the product of stacks.
The gluing maps \(gl_R(j, j')\) and \(gl_R^l(j,j')\) defined in Section~\ref{SSec:ModuliSpacesOfStableSUSYCurves} are morphisms in \(\cat{C_2}\).
Hence, by Proposition~\ref{prop:OperadConstruction}, one can construct an operad \(\cat{O_2}\colon \cat{SGr^{st}}\to \cat{C_2}\) such that
\begin{itemize}
	\item
		for any SUSY graph \(\tau\) \(\cat{O_2}(\tau) = \prod_{v\in V_\tau} \overline{\mathcal{M}}_{g(v), F_{\tau,NS}(v), F_{\tau, R}(v)}\),
	\item
		for any isomorphism \(h\colon \tau\to \sigma\) of SUSY corollas
		\begin{equation}
			\cat{O_2}(h) = a_{\left({\left(h^F|_{F_{\sigma, NS}}\right)}^{-1}, {\left(h^F|_{F_{\sigma, R}}\right)}^{-1}\right)},
		\end{equation}
	\item
		for any grafting \(h\colon \tau\to \sigma\), \(\cat{O_2}(h)\) is the identity,
	\item
		for any contraction \(con^l_{(f, f')}\colon \tau\to \tau\) of a loop at a SUSY corolla~\(\tau\) specified by \(f, f'\in F_{\tau}\) of the same color \(c_\tau(f)\in \Set{NS, R}\), \(\cat{O_2}(con^l_{(f, f')})= gl^l_{c_\tau(f)}(f, f')\),
	\item
		for any contraction \(con_{(f, f')}\colon \tau\to\sigma\) of the edge in a graph \(\tau\) consisting of two SUSY corollas \(\tau_1\), \(\tau_2\) grafted at the edge specified by \((f, f')\in F_{\tau_1}\times F_{\tau_2}\) of the same color \(c_\tau(f)\in \Set{NS, R}\), \( \cat{O_2}(con_{(f, f')})= gl_{c_\tau(f)}(f, f')\).
\end{itemize}

However, the category \(\cat{C_2}\) does not allow for a functor \(\cat{\Pi}\colon \cat{C_{2,red}}\to \cat{M}\).
Recall that the reduction functor \(\cat{Red}\) associates to each superstack \(M\) its reduced stack \(M_{red}\) and to each map of superstacks \(f\colon M\to M'\) its reduced maps \(f_{red}\colon M_{red}\to M'_{red}\).
This reduction can also be applied to weak maps in the sense of Definition~\ref{defn:WeakMaps} and also yields a functor \(\cat{Red}\colon \cat{C_2}\to \cat{C_{2,red}}\).
We have discussed in Section~\ref{SSec:ModuliSpacesOfStableSUSYCurves} that \({\left(\overline{\mathcal{M}}_{g, I_{NS}, I_R}\right)}_{red}=\overline{M}^{spin}_{g, I_{NS}, I_R}\) is a bundle over \(\overline{M}_{g, I_{NS}\cup I_R}\) with discrete fibers.
But not every map between bundles is a bundle map, that is, not every map \(\overline{M}^{spin}_{g,I_{NS}, I_R}\to \overline{M}^{spin}_{g', I'_{NS}, I'_R}\) descends to a map \(\overline{M}_{g, I_{NS}\cup I_R}\to \overline{M}_{g', I'_{NS}\cup I'_R}\).
This is fixed by the following definition:
\begin{defn}
	A weak map \(f\colon P\to M'\) between \(M\) and \(M'\in \cat{C_1}\) is said to project if there is a map \(\cat{\Pi}f\colon \cat{\Pi} M_{red}\to \cat{\Pi} M'_{red}\) such that the following diagram commutes:
	\begin{diag}
		\matrix[dr](m){
			& P_{red} &\\
			M_{red} & & M'_{red}\\
			\cat{\Pi} M_{red} & & \cat{\Pi} M'_{red}\\
		} ;
		\path[pf]{
			(m-1-2) edge (m-2-1)
				edge node[auto]{\(f_{red}\)} (m-2-3)
			(m-2-1) edge node[auto]{\(\pi\)} (m-3-1)
			(m-2-3) edge node[auto]{\(\pi\)} (m-3-3)
			(m-3-1) edge node[auto]{\(\cat{\Pi}f\)} (m-3-3)
		};
	\end{diag}
	Here \(\cat{\pi}\colon M_{red} \to \cat{\Pi} M_{red}\) (respectively \(\pi\colon M'_{red}\to \cat{\Pi} M'\)) is the projection map from the product of spin moduli spaces \(M_{red}\) to the product of moduli spaces \(\cat{\Pi}M_{red}\) obtained by forgetting the spin structures.

	We denote by \(\cat{SM}\) the category whose objects are objects of \(\cat{C_1}\), that is finite products of moduli superstacks.
	Morphisms of \(\cat{SM}\) are morphism from \(\cat{C_2}\) that project.
	We denote the corresponding reduction functor by \(\cat{Red}\colon \cat{SM}\to \cat{SM_{red}}\) and the projection functor by \(\cat{\Pi}\colon \cat{SM_{red}}\to \cat{M}\).
\end{defn}
The category \(\cat{SM}\) is a symmetric monoidal category and contains the isomorphisms of moduli superstacks induced by relabeling as well as all gluing maps of both Neveu--Schwarz and Ramond type as was discussed in Section~\ref{SSec:ModuliSpacesOfStableSUSYCurves}.
Consequently, one can restrict the operad \(\cat{O_2}\) to an operad \(\cat{O}\colon \cat{SGr^{st}}\to \cat{SM}\) as requested in Theorem~\ref{thm:SUSYOperad}.

The functors \(\cat{Red}\colon \cat{SM}\to \cat{SM_{red}}\) and \(\cat{\Pi}\colon \cat{SM_{red}}\to \cat{M}\) are symmetric monoidal functors.
Hence we obtain a commutative diagram of symmetric monoidal functors:
\begin{diag}\label{diag:SupermodularVsModularOperad}
	\matrix[mat](m){
		\cat{SGr^{st}} & \cat{SM} \\
			& \cat{SM_{red}}\\
		\cat{MGr^{st}} & \cat{M} \\
	} ;
	\path[pf]{
		(m-1-1) edge node[auto]{\(\cat{O}\)} (m-1-2)
			edge node[auto]{\(\cat{F}\)} (m-3-1)
			edge node[auto]{\(\cat{O_{red}}\)} (m-2-2)
		(m-1-2) edge node[auto]{\(\cat{Red}\)} (m-2-2)
		edge [bend left=70] node[auto]{\(\cat{P}\)} (m-3-2)
		(m-2-2) edge node[auto]{\(\cat{\Pi}\)} (m-3-2)
		(m-3-1) edge node[auto]{\(\cat{o}\)} (m-3-2)
	};
\end{diag}
This completes the proof of Theorem~\ref{thm:SUSYOperad}.

Note that \(\cat{O_{red}}=\cat{Red}\circ\cat{O}\) is the previously unknown SUSY operad of moduli stacks of spin curves.

As an application, we use the operad \(\cat{O}\) to calculate the dimension of the (compactified) moduli space of SUSY curves with prescribed dual graph \(\tau\).
Let \(\tau\) be a connected SUSY graph and let \(c\colon \tau\to \sigma\) be the full contraction.
That is, \(\sigma\) is a corolla of genus \(g(\tau)\) with Neveu--Schwarz flags \(T_{\tau, NS}\) and Ramond flags \(T_{\tau, R}\).
Then \(\cat{O}(c)\) is given by a diagram of the form
\begin{diag}
	\matrix[mat, column sep=large](m){
		P & \overline{\mathcal{M}}_{g(\tau), T_{NS}, T_R}\\
		\prod_{v\in V_\tau} \overline{\mathcal{M}}_{g(v), F_{\tau, NS}(v), F_{\tau, R}(v)} &\\
	} ;
	\path[pf]{
		(m-1-1) edge node[auto]{\(gl\)} (m-1-2)
			edge (m-2-1)
	};
\end{diag}
The image of \(gl\) in \(\overline{\mathcal{M}}_{g(\tau), T_{\tau, NS}, T_{\tau, R}}\) is \(\overline{\mathcal{M}}_\tau\), the closed moduli space of SUSY curves of dual graph \(\tau\).
Here \(\overline{\mathcal{M}}_\tau\) is described as product over the moduli spaces of the vertices together with the gluing data contained in \(P\).

We can use the dimension of the bundle \(P\) to calculate the dimension of \(\overline{\mathcal{M}}_\tau\).
The fiber dimension of \(P\) is \(0|\#E_{\tau, R}\).
The base space has the even dimension
\begin{equation}
	\sum_{v\in V_\tau}\left(3g_\tau(v) - 3 + \#F_{\tau, NS}(v) + \#F_{\tau, R}(v)\right)
\end{equation}
which coincides with the even dimension of \(P\).
Using the formula for the genus of the graph and that the number of flags coincides with the number of markings plus twice the numbers of edges
\begin{align}
	\sum_{v\in V_\tau} \left(g_\tau(v) - 1\right) &= g_\tau - 1 - \#E_\tau \\
	\sum_{v\in V_\tau} \#F_{\tau, NS}(v) &= \#F_{\tau, NS} = \#T_{\tau, NS} + 2\#E_{\tau, NS} \\
	\sum_{v\in V_\tau} \#F_{\tau, R}(v) &= \#F_{\tau, R} = \#T_{\tau, R} + 2\#E_{\tau, R}
\end{align}
Hence for the even dimension of \(P\) we have
\begin{equation}
	\begin{split}
		d_e(P) &= \sum_{v\in V_\tau}\left(3g_\tau(v) - 3 + \#F_{\tau, NS}(v) + \#F_{\tau, R}\right) \\
		&= 3g_\tau - 3 - 3\#E_\tau + \#T_{\tau, NS} + 2\#E_{\tau, NS} + \#T_{\tau, R} + 2\#E_{\tau, R} \\
		&= 3g_\tau - 3 + \#T_\tau - \#E_\tau
	\end{split}
\end{equation}
Similarly for the odd dimension of \(P\), taking into account the fiber dimension
\begin{equation}
	\begin{split}
		d_o(P) &= \sum_{v\in V_\tau}\left(2g_\tau(v) - 2 + \#F_{\tau, NS}(v) + \frac12\#F_{\tau, R}\right) + \#E_{\tau, R}\\
		&= 2g_\tau - 2 - 2\#E_\tau + \#T_{\tau, NS} + 2\#E_{\tau, NS} + \frac12\#T_{\tau, R} + \#E_{\tau, R} + \#E_{\tau, R} \\
		&= 2g_\tau - 2 + \#T_{\tau, NS} + \frac12\#T_{\tau, R}
	\end{split}
\end{equation}
As all gluing maps are embeddings, it follows that \(\overline{\mathcal{M}}_\tau\subset \overline{\mathcal{M}}_{g_\tau, T_{\tau, NS}, T_{\tau, R}}\) is a subspace of codimension \(\#E_\tau |0\).

\printbibliography

\textsc{Enno Keßler\newline
Max-Planck-Institut für Mathematik,
Vivatsgasse 7,
53111 Bonn,
Germany}\newline \texttt{kessler@mpim-bonn.mpg.de}\newline

\textsc{Yuri I. Manin\newline
Max-Planck-Institut für Mathematik,
Vivatsgasse 7,
53111 Bonn,
Germany}\newline \texttt{manin@mpim-bonn.mpg.de}\newline

\textsc{Yingying Wu\newline
Center of Mathematical Sciences and Applications,
Harvard University,
20~Garden Street,
Cambridge, MA 02138,
USA}\newline
\texttt{ywu@cmsa.fas.harvard.edu} \newline

\end{document}